\documentclass{amsart}
\usepackage{amsmath,amsfonts,amssymb,amsbsy,eucal,mathrsfs,enumitem}
\usepackage[all]{xy}
\usepackage{pstricks}
\usepackage{pst-plot}
\usepackage{graphicx}
\usepackage[utf8]{inputenc}
\usepackage{tikz,tikz-cd}
\usepackage{multicol}

\definecolor{pink}{rgb}{1,0.5,0.5}

\def\newthm#1#2{\newtheorem{#1}[dummy]{#2}%
  \expandafter\def\csname#2\endcsname##1{\hyperref[#1:##1]{#2~\ref*{#1:##1}}}}
\newthm{thm}{Theorem}
\newthm{lemma}{Lemma}
\newthm{prop}{Proposition}
\newthm{cor}{Corollary}
\newthm{conj}{Conjecture}
\newthm{cons}{Consequence}
\newthm{fact}{Fact}
\newthm{obs}{Observation}
\newthm{guess}{Guess}
\theoremstyle{definition}
\newthm{defn}{Definition}
\newthm{defi}{Definition}
\newthm{remark}{Remark}
\newthm{example}{Example}
\newthm{question}{Question}
\newthm{notation}{Notation}

\newcommand{\Section}[1]{\hyperref[sec:#1]{Section~\ref*{sec:#1}}}
\newcommand{\Table}[1]{\hyperref[tab:#1]{Table~\ref*{tab:#1}}}
\newcommand{\eqn}[1]{\hyperref[eqn:#1]{(\ref*{eqn:#1})}}

\DeclareMathOperator{\QH}{QH}

\newcommand{\fonction}[5]{
\begin{array}[t]{rrcll}
#1 & : & #2 & \rightarrow & #3 \\
   &   & #4 & \mapsto     & #5
\end{array}  }

\renewcommand{\P}{{\mathbb P}}
\newcommand{\C}{{\mathbb C}}

\newcommand{\Z}{{\mathbb Z}}

\newcommand{\cG}{{\mathcal G}}

\newcommand{\cP}{{\mathcal P}}

\newcommand{\ignore}[1]{}

\def\N{{\mathbb N}}

\def \a {{\alpha}}

\def \A {{\mathbb{A}}}

\def \im {{\rm Im}}

\def \Waff {{{W_{{\rm aff}}}}}

\def \Wafft {{\widetilde{W}_{{\rm aff}}}}

\def \Waffm {{W_{{\rm aff}}^-}}
\def \Wpaff {{{(W^P)}_{{\rm aff}}}}
\def \Wpafft {{{(\widetilde{W}^P)}_{{\rm aff}}}}
\def \Wafftm {{\widetilde{W}_{{\rm aff}}^-}}

\def \xh {{\widehat{x}}}
\def \yh {{\widehat{y}}}

\newcommand{\scal}[1]{\langle #1 \rangle}

\newcommand{\fh}{\widehat{{f}}}

\newcommand{\h}{\mathfrak h}

\newcommand{\ad}{{\rm ad}\ \!}

\newcommand{\comment}[1]{}

 \def\cG{\mathcal{G}}

\def\cP{\mathcal{P}}

\def\fg{\mathfrak{g}}
 \def\fh{\mathfrak{h}}

\def\HT{{H_*^T}}
\def\QHT{{QH_T^*}}
\def\FS{{\rm{Frac}(S)}}
\def\jad{{j^{\rm ad}}}

\def\wh{{\widehat w}}

\def\psit{{\widetilde \psi}}

\def\xit{{\widetilde \xi}}

\def\At{{\widetilde A}}

\def\Iaff{{I_{\rm aff}}}

\def\Jt{{\widetilde J}}
\def\Mt{{\widetilde M}}

\def\Qv{{Q^\vee}}
\def\Qvm{{Q^\vee_-}}
\def\Qvp{{Q^\vee_P}}
\def\Pv{{P^\vee}}

\def\OK{{\Omega K}}
\def\OKad{{\Omega K^{\rm ad}}}
\def\Gad{{G^{\rm ad}}}
\def\Kad{{K^{\rm ad}}}

\def\Aaff{{\mathbb{A}_{\rm aff}}}
\def\Aafft{{\widetilde{\mathbb{A}}_{\rm aff}}}
\def\Qaff{{Q_{\rm aff}}}

\def\Qafft{{\widetilde{Q}_{\rm aff}}}
\def\Raff{{R_{\rm aff}}}
\def\Waff{{W_{\rm aff}}}
\def\Wafft{{\widetilde{W}_{\rm aff}}}
\def\Wafftm{{{\Wafft^-}}}
\def\Wpaff{{(W^P)_{\rm aff}}}
\def\Wpafft{{(\widetilde{W}^P)_{\rm aff}}}

\def\fgaff{\fg_{\rm aff}}
\def\fhaff{\fh_{\rm aff}}

\def\av{{\alpha^\vee}}
\def\bv{{\beta^\vee}}

\def\lv{{\lambda^\vee}}
\def\mv{{\mu^\vee}}
\def\nv{{\nu^\vee}}
\def \Frac{{{\rm Frac}}}

\begin{document}

\title{{Affine symmetries in quantum cohomology: corrections and new results}}

\date{June 18, 2020}

\author{P.--E. Chaput}
\address{Domaine Scientifique Victor Grignard, 239, Boulevard des
  Aiguillettes, Universit{\'e} Henri Poincar{\'e} Nancy 1, B.P. 70239,
  F-54506 Vandoeuvre-l{\`e}s-Nancy Cedex, France}
\email{pierre-emmanuel.chaput@univ-lorraine.fr}

\author{N. Perrin}
\address{Université Paris-Saclay, UVSQ, CNRS, Laboratoire de Mathématiques de Versailles, 78000, Versailles, France.
}
\email{nicolas.perrin@uvsq.fr}

\subjclass[2000]{14M15, 14N35}

\begin{abstract}
  In \cite{cmp4} a general formula was given for the multiplication
  by some special Schubert classes
  in the quantum cohomology of any homogeneous space. Although this formula is true in the non
  equivariant setting, the stated equivariant version was wrong. We provide correction for the
  equivariant formula, thus giving a correct argument for the non equivariant formula. We also give new formulas in the equivariant homology of the affine grassmannian that could lead to non-equivariant Pieri formulas.
\end{abstract}

\maketitle

\markboth{P.-E.~CHAPUT, N.~PERRIN}
{\uppercase{{Affine symmetries in quantum cohomology: erratum}}}

\setcounter{tocdepth}{2}

\section{Introduction}

In \cite{cmp4} a general formula was given for the multiplication in the quantum cohomology of any homogeneous space by some special Schubert classes coming from cominuscule weights. Although this formula is true in the non equivariant setting, the stated equivariant version is wrong. We provide correction for the equivariant formula, thus giving a correct argument for the non equivariant formula. We also provide new product formulas in the equivariant homology of the affine grassmannian.

Let $G$ be a semisimple simply connected algebraic group and fix $T \subset B$ a maximal torus and a Borel subgroup containing it. Denote by $P^\vee$ and $Q^\vee$ be the coweight and corrot lattices. A dominant coweight $\lambda^\vee \in P^\vee$ is minuscule if $\scal{\lambda^\vee,\alpha} \in \{0,1\}$ for any positive root $\alpha$. A minuscule dominant coweight is a fundamental coweight. Denote by $I_m$ the subset of the set $I$ of vertices of the Dynkin diagram of $G$ parametrising minuscule coweights.

We consider a finite group $Z$ which has several interpretation. Define $Z$ has
$$Z := P^\vee/Q^\vee.$$
Representatives for this quotient are for example the opposites of the minuscule fundamental coweights $(-\varpi^\vee_i)_{i \in I_m}$. The group $Z$ is also the center of $G$ and if $G^\ad$ the the adjoint group associated to $G$, then $Z = \pi_1(G^\ad)$.

The group $Z$ can be realised as a subgroup of the Weyl group $W$ of $G$ as follows. Let $w_0$ be the longest element in $W$. For $i \in I_m$ define $v_i \in W$ to be the smallest element in $W$ such that $v_i\varpi_i = w_0\varpi_i$. Then the family $(v_i)_{i \in I_m}$ forms a finite subgroup of $W$ isomorphic to $Z$. Finally $Z$ can be realised as a subgroup of the extended affine Weyl group $\Waff = W \ltimes P^\vee$ (see Section \ref{subsec:affineweyl} below) by $i \mapsto \tau_i := v_i t_{-\varpi_i}$.

\medskip

For $P \subset G$ a parabolic subgroup, let $I_P$ be the set of vertices in the Dynkin diagram such that, for $i \in I$, the simple root $\alpha_i$ is a root of $P$ if and only if $i \in I_P$. For $w \in W$, denote by $\sigma^P(w)$ the Schubert class in $H^{2 \ell(w)}(G/P,\Z)$ defined by $w$. Denote by $Q_P^\vee$ the coroot lattice of $P$ and consider $\eta_P : Q^\vee \to Q^\vee/Q_P^\vee$ the quotient map. We define an action of the Weyl group $W$ of $G$ on the equivariant cohomology $H^*_T(G/P)$ using, for $w \in W$, the pull-back in cohomology of the left multiplication by $w$ (see Subsection \ref{subsection-peterson}). We denote this action by $w^*$. This action is trivial in non-equivariant cohomology and extends to an action on equivariant quantum cohomology  $\QH_T^*(G/P)$. In this paper we obtain the following formula in the quantum equivariant cohomology $\QH^*_T(G/P)$ for any parabolic subgroup $P \subset G$ (see Theorem \ref{theo}).

\begin{thm}
Let $i$ be a cominuscule node. In $QH^*_T(G/P)$ we have
$$
\sigma^P(v_i) \times v_{i}^*(\sigma^P(w)) = q_{\eta_P(\varpi_i^\vee - w^{-1}(\varpi_i^\vee))}
\sigma^P(v_iw).
$$
\end{thm}

This result corrects our formula in \cite[Theorem 1]{cmp4} which was wrong in the equivariant setting (the action $v_i^*$ on the second factor on the LHS was missing). The error in \cite{cmp4} comes from an incorrect description of the ring structure of $H_*^T(\OKad)$ the equivariant homology of the adjoint affine grassmannian (see Section \ref{section:homology}). If $\OK$ is the affine grassmannian for $G$, the incorrect claim (\cite[Page 12]{cmp4}) was that $H_*^T(\OK)$ should be isomorphic to $Z \otimes H_*^T(\OK)$. This is not true as explained in Section \ref{section:homology} (see Remark \ref{err1}). This is corrected in the present paper. Especially, in Proposition \ref{prop-HOKad-algebre}, we prove the $S$-algebra isomorphism (here $S = H^*_T({\rm pt})$):
$$H_*^T(\OKad) \simeq S[P^\vee] \otimes_{S[\Qv]} H_*^T(\OK).$$
The incorrect product formula was then used only once in \cite[Proposition 3.16]{cmp4}. We give a correct version of Proposition 3.16 in \cite{cmp4} in Proposition \ref{productpip}.

We tried to write this paper as independently from \cite{cmp4} as possible and included many preliminary results on the algebra and the module structure of the extended affine Hecke algebra $\Aafft$ (see Section \ref{nil-Hecke}) and on its module structure $\Mt$ which is isomorphic to $H_*^T(\OKad)$ the homology of the adjoint affine grassmannian. We also added new results. Especially we provide a generalization of a formula in  \cite[Proposition 5.4]{lam} to coweights for the map $j^\ad : H_*^T(\OKad) \to Z_{\Aafft}(S)$ (see Proposition \ref{jadtranslation}).

\begin{prop}
Let $\mv \in P^\vee$ be antidominant and set $W_\mv = \scal{s_{\alpha_i} \ | \ i \in [1,r] \textrm{ and } \scal{\alpha_i,\mv} = 0} = \{ w \in W \ | \ w(\mv) = \mv \}$. Then
$$\jad(\xi_{t_\mv})=\sum_{w \in W/W_\mv} \At_{t_{w(\mv)}}\, .$$
\end{prop}

Finally, we use this formula to give an explicit formula for the image of the map $j : H_*^T(\OK) \to Z_{\Aafft}(S)$ for the special elements $\tau_i(v_i) = \tau_i v_i \tau_i^{-1}$ (see Proposition \ref{prop_pieri}). Here $\Wafftm$ denotes the set of minimal representatives of the quotient $\Wafft/W$.

\begin{prop}
  We have $\tau_i(v_i) \in \Wafftm$ and
  $$j(\xi_{\tau_i(v_i)}) = \sum_{w \leq_L v_i} \sum_{v \leq v_I^{-1}} \tau_i(\xi^v(v_i^{-1})) A_{\tau_i(w)v_i w^{-1}} A_v,$$
  where $\leq$ is the Bruhat order and $\leq_L$ the weak left Bruhat order.
\end{prop}

We hope to use the above formula to prove Pieri type formulas in $H_*^T(\OK)$ in the spirit of what Lam, Lapointe, Morse and Shimozono \cite{llms} did in type $A$.

\medskip
\noindent
{\bf Acknowledgement:} we thank Elizabeth Milićević for showing us a counterexample to \cite[Theorem 1]{cmp4} which lead to the
present correction and development.

\tableofcontents

\section{Notations}

In this section, we fix notation for affine Kac-Moody Lie algebras, we introduce the finite group $Z$ and define the extended affine Weyl group $\Wafft$.

\subsection{Affine Lie algebras}

We denote by $\fg$ a simple finite-dimensional Lie algebra of rank $r$, and by $\fh$ a Cartan subalgebra.
We denote by $G$ the simply-connected group corresponding to $\fg$ and by $\Gad$ the adjoint group.
The affine Kac-Moody group corresponding to $G$ will be denoted by $\cG$ and $\cP \subset \cG$ is the parabolic subgroup such that
$\cG/\cP$ is the affine Grassmannian.

The corresponding affine Lie algebra will be denoted by $\fgaff$, with Cartan subalgebra $\fhaff$. The simple roots are denoted $(\alpha_i)_{i \in [1,r]}$ and the null-root, orthogonal to all the simple roots $(\alpha_i)_{i \in [1,r]}$, will be denoted by $\epsilon$. Recall that we have the equality $\epsilon = \Theta + \alpha_0$, where $\Theta$ is the highest root of $\fg$.
As in \cite[p.82]{kac} we will use the decompositions
$\fhaff^\vee = \fh^\vee \oplus \C \Lambda_0 \oplus \C \epsilon$ and
$\fhaff=\fh \oplus \C K \oplus \C d$.
We denote by $\Raff$ the set of roots of $\fgaff$ and by $R$ those of $\fg$.

We denote by $Q,P,\Qv,\Pv$ the root, weight, coroot, coweight lattices of $\fg$.
We also denote by $S$ the symmetric algebra on $P$.

\subsection{Affine Weyl groups}
\label{subsec:affineweyl}

Let $W$ be the Weyl group of $\fg$ and
let $\Waff = Q^\vee \rtimes W$ be the affine Weyl group.
For $\lambda^\vee \in Q^\vee$, the corresponding element in $\Waff$ will be denoted by
$t_{\lambda^\vee}$.
The reflection associated to a root $\alpha$ will be denoted by $s_\alpha$.
The group $\Waff$ is a Coxeter group with Coxeter generators
$s_i$ for $1 \leq i \leq r$ and $s_0 = t_{\Theta^\vee} s_{\Theta}$
(\cite[Prop 13.1.7]{kumar}, see also Lemma \ref{talpha}).

Define the extended affine Weyl group $\Wafft := P^\vee \rtimes W \supset \Waff$. The group $\Waff$ acts on $P \oplus \Z \epsilon$ while the group by $\Wafft$ acts only on $Q \oplus \Z \epsilon$ via
\begin{equation}
\label{defaction}
\begin{array}{c}
w t_\lv \cdot (\mu + n\epsilon) = w(\mu) + (n - \scal{\mu,\lv}) \epsilon \\
w t_\mv \cdot (\lambda + n\epsilon) = w(\lambda) + (n- \scal{\lambda,\mv}) \epsilon
\end{array}
\end{equation}
where we have $\lambda \in Q,\lv \in \Qv,\mu \in P,\mv \in \Pv$. Note that in general $\Wafft$ does not act on $P \oplus \Z \epsilon$ since $\scal{P,P^\vee} \not \subset \Z$ in general.

We may however define an actions of $\Wafft$ on $P^\vee$ (and therefore on $Q^\vee$) by prescribing that translations do not act: we simply set $wt_\lv(\mv) = w(\mv)$ for $w \in W$ and $\lv,\mv \in P^\vee$.

\begin{notation}
Since an element in $\Qv$ is also an element in $\Qv \oplus \Z \delta$, we will denote by $w \cdot \lv$ the result of the
action of $w \in \Wafft$ on $\lv$ as an element in $\Qv \oplus \Z \delta$ and by $w(\lv)$ the element in $\Qv$.
\end{notation}

Recall the definition of the fundamental alcove
$$A_\circ = \{ \lambda \in \h_\mathbb{R}^\vee \ | \ \scal{\lambda,\alpha_i^\vee} \geq 0 \textrm{ for all } i \in [1,r] \textrm{ and } \scal{\lambda,\Theta^\vee} \leq 1 \}.$$
    The stabiliser of $A_\circ$ in $\Wafft$  will be denoted by $Z$; it is a subgroup of $\Wafft$ isomorphic to $\Pv/\Qv$ \cite[\S 10.1]{lam-shimozono}. In \emph{loc. cit.}, the authors also prove the following result.

\begin{lemma}
\label{lemma:actiontau}
Let $\tau \in Z$.
The conjugation by $\tau$ is an automorphism of the Coxeter group $\Waff$.
In fact, there exists an automorphism $f_\tau$ of the affine Dynkin diagram such that
$$\forall i \in I \cup \{0\}\ ,\ \tau s_{\alpha_i} \tau^{-1} = s_{\tau \cdot \alpha_i} = s_{\alpha_{f_\tau(i)}}\ .$$
In particular, we have $\tau \cdot \epsilon = \epsilon$.
\end{lemma}

\begin{notation}
For $\xh \in \Waff$, set $\tau(\xh) := \tau \xh \tau^{-1} \in \Waff$. We have $\ell(\tau(\xh)) = \ell(\xh)$.
\end{notation}

\begin{lemma}
\label{zlength0}
An element $\tau$ in $Z$ permutes the positive real roots.
\end{lemma}
\begin{proof}
According to Lemma \ref{lemma:actiontau}, we have
$\tau \cdot (\sum n_j \alpha_j + n\delta)=\sum n_j\alpha_{f_\tau(j)} + n \delta$.
Since a real root $\alpha + n\delta$ is positive if and only if $n>0$ or $n=0$ and $\alpha>0$, $\tau$ indeed permutes
positive roots.
\end{proof}

As explained in \cite{cmp4}, $\Wafft$ is not a Coxeter group, but we have a well defined length function.
\begin{defi}
\label{defi:xhat}
Every element $x \in \Wafft$ can be uniquely written as $\tau \xh$ with $\tau \in Z$ and
$\xh \in \Waff$
\begin{enumerate}
\item Define the length function by $\ell(x):=\ell(\xh)$.
\item Define a partial order on $\Wafft$ by $\displaystyle \tau \xh \leq \sigma \yh \Longleftrightarrow \tau=\sigma \mbox{ and } \xh \leq \yh.$
\end{enumerate}
Covering relations in $\Wafft$ for the above partial order are defined by $x \lessdot y$ if $x \leq y$ and $\ell(x)=\ell(y)-1$.
\end{defi}

\begin{remark}
\label{inversions}
The length of $x \in \Wafft$ is also the number of inversions, namely the cardinal of the set $I(x) = \{ \alpha \in \Raff \ | \ \alpha > 0 , \alpha \textrm{ is real and } x(\alpha) < 0 \}$.
Indeed, for $x=\tau \xh$, by Lemma \ref{zlength0}, we have $I (x) = I(\xh)$.
\end{remark}

\subsection{Translations}

We will need the following lemma.

\begin{lemma}
\label{talpha}
Let $\alpha \in R$. We have $t_{\av} = s_{\epsilon-\alpha} s_\alpha$.
\end{lemma}
\begin{proof}
Set $K^\bot = \{ \mu \in \fhaff^\vee \ | \ \scal{\mu,K} = 0\}$. By \cite[p.87]{kac}, it is enough to compute
$s_{\epsilon-\alpha}s_{\alpha}(\mu)$ for $\mu \in K^\bot$.
We have
$$
s_{\epsilon-\alpha}s_\alpha(\mu)
= \mu - \scal{\mu,(\epsilon-\alpha)^\vee}(\epsilon-\alpha)
- \scal{\mu,\av} \alpha
+ \scal{\mu,\av} \scal{\alpha,(\epsilon-\alpha)^\vee}(\epsilon-\alpha).
$$
Now, for any $\beta \in \Raff$, we have by \cite[\S 2.3.5 and \S 6.2.3]{kac}:
$$
\scal{\mu,(\epsilon+\beta)^\vee}
= \frac{2(\epsilon+\beta,\mu)}{(\epsilon+\beta,\epsilon+\beta)}
= \frac{2(\beta,\mu)}{(\beta,\beta)}
= \scal{\mu,\bv}.
$$
Therefore,
$$
\begin{array}{rcl}
s_{\epsilon-\alpha}s_\alpha(\mu) & = &
\mu + \scal{\mu,\alpha^\vee}(\epsilon-\alpha)
- \scal{\mu,\av} \alpha
- 2\scal{\mu,\av} (\epsilon-\alpha) \\
& = & \mu - \scal{\mu,\av} \epsilon = t_\av(\mu)\ ,
\end{array}
$$
where the last equality follows from the definition of $t_\av$ in \cite[\S 6.5.5]{kac}.
\end{proof}

\begin{cor}
\label{salphaepsilon}
For $\alpha \in R,k \in \Z$ and $\mv \in P^\vee$, we have $ s_{\alpha+k\epsilon}(\mv) = s_{\alpha}(\mv)$. 
\end{cor}

\begin{proof}
  We have $s_{\alpha + \epsilon}(\mv) = s_{\alpha + \epsilon} t_{-\alpha}(\mv) = s_{\alpha + \epsilon} s_{\alpha + \epsilon} s_{-\alpha} (\mv) = s_{-\alpha}(\mv) = s_\alpha(\mv)$. The result follows by induction.
\end{proof}

\section{extended nil-Hecke ring}
\label{nil-Hecke}

The goal of this section is to extend the notion of
the nil-Hecke ring defined by Kostant and Kumar \cite{kostant}.
This ring was used in \cite{lam-shimozono} to compare the quantum cohomology of $G/P$
and the homology of affine Grassmannians $\OK$. We need a refined version of this
nil-Hecke ring that enables dealing with $\OKad$ the adjoint affine grassmannian (see Section \ref{section:homology}).

\subsection{Definition}
We extend several classical object in particular the affine nil-Hecke algebra. Our reference for these classical objects is Kumar's book \cite{kumar}.

\begin{defi}
\label{defi:qaff}
Recall that the ring $\Qaff$ is
\begin{equation}
\label{equa:qaff}
\Qaff = \bigoplus_{w \in \Waff} \FS\, \delta_w
\end{equation}
We define the following extended version:
\begin{equation}
\label{equa:qafftilde}
\Qafft = \bigoplus_{w \in \Wafft} \FS\, \delta_w
\end{equation}
In both cases, the ring structure is defined by the equations
$\delta_u \delta_v = \delta_{uv}$ and $\delta_u s = u(s) \delta_u$, for $u,v \in \Waff$
resp. $u,v \in \Wafft$, and $s \in S$.
\end{defi}

\begin{defi}
  \label{defi:Aw}
  As in \cite{kumar}, we consider particular elements in $\Qaff$.
  \begin{enumerate}
    \item For $i \in I$, set $A_i = \frac{1}{\alpha_i}(\delta_e-\delta_{s_i})$.
      For $i=0$, set $A_0 = \frac{-1}{\Theta}(\delta_e-\delta_{s_0})$ (note that this is coherent with the forthcoming Definition \ref{defi:aalpha}).
 \item For $w \in \Waff$ and for $w=s_{i_1} \cdots s_{i_\ell}$ a reduced expression. We set:
$$
A_w = A_{i_1} \cdots A_{i_\ell}\, .
$$
  \end{enumerate}
By \cite[Theorem 11.1.2]{kumar}, the right hand side does not depend on the chosen reduced expression.
\end{defi}

Recall that for $x \in \Wafft$, there is a unique decomposition $x = \tau \xh$ with $\tau \in Z$ and $\xh \in \Waff$.

\begin{defi}
\label{defi:Ax}
Let $x = \tau \xh \in \Wafft$, we set
$\At_w = \delta_\tau A_\xh.$
\end{defi}

\begin{defi}
\label{defi:nilhecke}
As in \cite{kostant}, the (extended) nil-Hecke ring is generated over $S$ by the elements $A_w$.
\begin{enumerate}
\item The nil-Hecke ring is 
$$\Aaff = \bigoplus_{w \in \Waff} S\cdot A_w  \subset \Qaff \, .$$
\item The extended nil-Hecke ring is 
$$\Aafft = \bigoplus_{w \in \Wafft} S\cdot \At_w \subset \Qafft \, .$$
\end{enumerate}
\end{defi}

\begin{remark}
We will see below that both are indeed subrings of $\Qafft$.
\end{remark}

\subsection{Definition and properties of $A_\alpha$}

It will be helpful to generalize the definition of $A_i$ in the following way.

\begin{defi}
\label{defi:aalpha}
For a real root $\alpha = \gamma + k\epsilon$ with $\gamma \in R$, set
$A_\alpha = \frac{1}{\gamma}(\delta_e - \delta_{s_\alpha}).$
\end{defi}

These elements satisfy the following properties.

\begin{prop}
\label{prop:aallpha}
Let $w \in \Wafft, \alpha \in \Raff$ a real root and $\lambda \in Q$. Then we have:
\begin{enumerate}
 \item $\delta_w A_\alpha \delta_{w^{-1}} = A_{w(\alpha)}$.
 \item $A_\alpha \lambda = s_\alpha(\lambda) A_\alpha + \scal{ \lambda , \alpha^\vee } 1$.
\end{enumerate}
\end{prop}

\begin{proof}
Let $w=ut_\mv,\alpha=\gamma+k\epsilon,\lambda \in P$ be as in the proposition. Then,
$$
\begin{array}{rcccl}
\delta_w A_\alpha \delta_{w^{-1}} & = &
\delta_w \frac{1}{\gamma}(\delta_e-\delta_{s_\alpha}) \delta_{w^{-1}}
& = & \frac{1}{u(\gamma)} \delta_w (\delta_e-\delta_{s_\alpha}) \delta_{w^{-1}} \\
& = & \frac{1}{u(\gamma)}(\delta_e-\delta_{ws_\alpha w^{-1}}) & = & \frac{1}{u(\gamma)}(\delta_e-\delta_{s_{w(\alpha)}}) = A_{w(\alpha)}\, .
\end{array}
$$

For the second point, we use \cite[\S 6.1]{lam-shimozono} and the above conjugation relation. Let $w \in \Waff$ and
$i \in I \cup \{0\}$ be such that
$\alpha = w(\alpha_i)$.
We have
$$
\begin{array}{rcl}
A_\alpha \lambda & = & \delta_w A_i \delta_{w^{-1}} \lambda \\
& = & \delta_w A_i w^{-1}(\lambda) \delta_{w^{-1}} \\
& = & \delta_w(s_iw^{-1}(\lambda)A_i + \scal{w^{-1}(\lambda),\alpha_i^\vee}) \delta_{w^{-1}} \\
& = & ws_iw^{-1}(\lambda) \delta_w A_i \delta_{w^{-1}} + \scal{\lambda,w(\alpha_i)^\vee} \\
& = & s_\alpha(\lambda) A_\alpha + \scal{\lambda,\alpha^\vee} \, .
\end{array}
$$
\end{proof}

\begin{remark}
The second formula in the above proposition generalizes the usual relation satisfied by the elements $A_i$ (see for example \cite[\S 6.1]{lam-shimozono}).
\end{remark}

\begin{cor}
For any real root $\alpha$, we have $A_\alpha \in \Aaff$.
\end{cor}

\begin{cor}
The (extended) nil Hecke rings $\Aaff$ and $\Aafft$ are subrings of $\Qafft$.
\end{cor}

\begin{proof}
The second formula above shows that for $s,s' \in S$ and $u,v \in \Waff$, the product $s A_u s' A_v$ lies in $\Aaff$ therefore $\Aaff$ is a ring. The first formula proves that for $\tau,\sigma \in Z$ and for $u,v \in \Waff$, we have $\delta_\tau A_u \delta_\sigma A_v = \delta_\tau \delta_{\sigma} A_{\sigma^{-1}(u)}A_v \in \Aafft$ proving that $\Aafft$ is a ring.
\end{proof}

\section{Module and ring structures of $\Aafft$}

In this section we present three different descriptions of $\Aafft$ and describe its $S$-module structure and its ring structure in each case.

\subsection{$S$-module structure of $\Aafft$}

Recall that we have an injection of $\Waff$ in the group of invertibles of $\Aaff$, given by
$w \mapsto \delta_w$: in fact $\delta_{s_i} = 1 - \alpha_i A_i \in \Aaff$ thus $\delta_w \in \Aaff$ for all $w \in \Waff$. Therefore the subgroup $Q^\vee \subset \Waff$ also injects in $\Aaff$, and
since $\Aaff$ is a ring we have an injection of the Laurent polynomial algebra $\Z[\Qv]$ inside $\Aaff$.
Thus $\Aaff$ is $\Z[\Qv]$-module via left mutliplication. The natural $\Z$-module basis of
$\Z[\Qv]$ will be denoted by $(h_\lv)_{\lv \in \Qv}$.

We now introduce two new algebraic models of $\Aafft$.

\begin{defi}
\label{defi:deuxdef}
Let $\varphi_1,\varphi_2$ be the following morphisms of $\Z$-modules:
$$
\fonction{\varphi_1}{\Z[\Pv] \otimes_{\Z[\Qv]} \Aaff}{\Qafft}{h_{\lv} \otimes A_w}{\delta_{t_\lv} A_w\, ,}
$$
$$
\fonction{\varphi_2}{\Z[Z] \otimes_\Z \Aaff}{\Qafft}{\tau \otimes A_w}{\delta_\tau A_w\, .}
$$
Note that $\Aaff$ has a structure of $S$-bimodule, thus also the two tensor products in this definition.
Both maps
$\varphi_1$ and $\varphi_2$ are $S$-linear on the right,
moreover $\varphi_1$ is also $S$-linear on the left
whereas $\varphi_2$ is not.
\end{defi}

\begin{prop}
\label{prop:troisdef}
With the above notations, $\im(\varphi_1)=\im(\varphi_2)=\Aafft$. Moreover, if $J \subset \Aaff$ is a left ideal, then
$\varphi_1(\Z[\Pv] \otimes_{\Z[\Qv]} J) = \varphi_2(\Z[Z] \otimes_\Z J)$.
\end{prop}
\begin{proof}
Observe that $\varphi_1$ is well-defined: $\varphi_1(h_\lv \otimes 1) = \varphi_1(1 \otimes \delta_{t_\lv}) = \delta_{t_\lv}$ for $\lv \in \Qv$.
We now prove that $\varphi_1(\Z[\Pv] \otimes J) \subset \varphi_2(\Z[Z] \otimes J)$.
Let $\lv \in \Pv$: there exists $\tau \in Z$
and $\wh \in \Waff$ such that
$t_\lv = \tau \wh$. Then for $a \in J$, we have $\varphi_1(h_\lv \otimes a) = \delta_{t_\lv}a = \delta_\tau \delta_\wh a \in \varphi_2(\Z[Z] \otimes J)$ since
$\delta_\wh a \in J$.

The reverse inclusion $\varphi_2(\Z[Z] \otimes J) \subset \varphi_1(\Z[\Pv] \otimes J)$ follows similarly from the fact
that any element in $Z$ can be written as a product $t_\lv u$ for some $\lv \in \Pv$ and $u \in W$. Finally, the equality
$\im(\varphi_2)=\Aafft$ follows from the definition of $\Aafft$ (Definition \ref{defi:nilhecke}).
\end{proof}

\subsection{Ring structure of $\Aafft$}

We give the description of the ring structure of $\Aafft$ according to the given three equivalent definitions of this module.

\begin{prop}
\label{product1}
Let $x,y \in \Wafft$, then we have
$$
\At_x \At_y =
\left \{
\begin{array}{l}
\At_{xy} \mbox{ if } \ell(xy) = \ell(x) + \ell(y) \\
0 \mbox{ otherwise}\, .
\end{array}
\right .
$$
\end{prop}
\begin{proof}
  Write $x = \sigma \xh$ and $y = \tau \yh$ with $\sigma,\tau \in Z$ and $\xh,\yh \in \Waff$. Recall that for $u,v \in \Waff$, we have:
$$
A_u A_v =
\left \{
\begin{array}{l}
A_{uv} \mbox{ if } \ell(uv) = \ell(u) + \ell(v) \\
0 \mbox{ otherwise}\, .
\end{array}
\right .
$$
By Lemma \ref{lemma:actiontau}, we have
$$
\begin{array}{rcl}
\At_x \At_y & = & \delta_\sigma A_\xh \delta_\tau A_\yh \\
& = & \delta_\sigma \delta_\tau A_{\tau^{-1}(\xh)} A_\yh \\
& = &
\left \{
\begin{array}{l}
\delta_{\sigma \tau} A_{\tau^{-1}(\xh)\yh} \mbox{ if } \ell(\tau^{-1}(\xh)\yh) = \ell(\tau^{-1}(\xh)) + \ell(\yh) \\
0 \mbox{ otherwise}
\end{array}
\right .
\end{array}
$$
But $\ell(\tau^{-1}(\xh) \yh) = \ell(xy)$ since $xy = \sigma \xh \tau \yh = \sigma\tau \tau^{-1}(\xh)\yh$,
and $\ell(\tau^{-1}(\xh)) + \ell(\yh) = \ell(\xh) + \ell(\yh) = \ell(x) + \ell(y)$. The result follows.
\end{proof}

\medskip

We now express the product in $\Aafft = \varphi_1(\Z[\Pv] \otimes_{\Z[\Qv]} \Aaff)$. Note that we need to compute the product $(\delta_{t_\lv} A_u)(\delta_{t_\mv} A_v)$. We therefore need to ``move'' $\delta_{t_\mv}$ to the left of $A_u$. The following proposition gives formulas for this.

\begin{prop}
\label{product2}
Let $\lv \in \Pv$ and let $\alpha=\gamma + k\epsilon \in \Raff$. Then:
\begin{enumerate}
\item $A_\alpha \delta_{t_{\lv}} = \delta_{s_\alpha(t_\lv)}A_\alpha +
     \frac{1}{\gamma}(\delta_{t_\lv}-\delta_{t_{s_\alpha(\lv)}})$.
\item $\delta_{t_\lv}-\delta_{t_{s_\alpha(\lv)}} = \delta_{t_\lv}(1-\delta_{t_{-\scal{\alpha,\lv}\av}}),$
\item $1-\delta_{t_{n\av}} = (1+\delta_{t_\av}+\cdots+\delta_{t_{(n-1)\av}})(1-\delta_{t_\av})\, \textrm{ for } n \in \N, $
 \item $1-\delta_{t_{-\av}} = \gamma (A_\alpha - \gamma A_\alpha A_{\epsilon - \alpha} + A_{\epsilon-\alpha})$.
\end{enumerate}
\end{prop}

\begin{proof}
  (1) From the equality $s_\alpha t_\lv = t_{s_\alpha(\lv)}s_\alpha$, we get $\delta_{s_\alpha} \delta_{t_\lv} = \delta_{t_{s_\alpha(\lv)}} \delta_{s_\alpha}$ in $\Aafft$. By definition of
$A_\alpha$ (Definition \ref{defi:aalpha}), this relation implies
  $$(1-\gamma A_\alpha) \delta_{t_\lv} = \delta_{t_{s_\alpha(\lv)}} (1-\gamma A_\alpha)\, .
$$
Thus we get $\gamma A_\alpha \delta_{t_\lv} = \gamma \delta_{t_{s_\alpha(\lv)}} A_\alpha + \delta_{t_\lv} - \delta_{t_{s_\alpha(\lv)}}$.

\noindent
(2) and (3) are easy consequences of the product formulas in $\Qafft$.

\noindent
(4) By Lemma \ref{talpha}, we have
$$
\begin{array}{rcl}
1-\delta_{t_{-\av}} & = & 1-\delta_{s_\alpha s_{\epsilon-\alpha}} \\
& = & 1-(1-\gamma A_\alpha)(1+\gamma A_{\epsilon-\alpha}) \\
& = & \gamma A_\alpha - \gamma A_{\epsilon-\alpha} + \gamma A_\alpha \gamma A_{\epsilon-\alpha} \\
& = & \gamma A_\alpha - \gamma A_{\epsilon-\alpha} +
\gamma \left ( s_\alpha(\gamma) A_\alpha + \scal{\av,\gamma} \right ) A_{\epsilon-\alpha} \\
& = & \gamma A_\alpha + \gamma A_{\epsilon-\alpha} - \gamma^2 A_\alpha A_{\epsilon-\alpha}\, ,
\end{array}
$$
where we used Proposition \ref{prop:aallpha} on the fourth line.
\end{proof}

\medskip

The ring structure in $\Aafft=\varphi_2(\Z[Z] \otimes_\Z \Aaff)$ is easy to describe:

\begin{prop}
\label{product3}
Let $\sigma,\tau \in Z$ and let $a,b \in \Aaff$. Then:
$$
\varphi_2(\sigma \otimes a) \cdot \varphi_2(\tau \otimes b)
= \varphi_2(\sigma \tau \otimes \tau^{-1}(a)b)\, .
$$
\end{prop}
\begin{proof}
This follows from the fact that in $\Aafft$, we have
$\delta_\sigma a \delta_\tau b = \delta_\sigma \delta_\tau \tau^{-1}(a)b \, .$
\end{proof}

\medskip

In the next proposition, we give an explicit formula for the commuting relation of the elements $\At_x$ and $\lambda \in P$, generalizing \cite[Proposition 4.3.b]{kostant}:

\begin{prop}
\label{product4}
Let $x \in \Wafft$ and let $\lambda \in P$. We have:
$$
\At_x \lambda = x(\lambda) \At_x + \sum_{\a:\ xs_\a \lessdot x} \scal{\lambda,\a^\vee} \At_{xs_\a}\, ,
$$
where the sum runs over positive real roots $\a$ such that $xs_\a \lessdot x$.
\end{prop}

\begin{proof}
Let $x=\tau \xh \in \Wafft$ with $\tau \in Z$ and $\xh \in \Waff$. Let $\lambda \in P$. According to Definition \ref{defi:Ax}, we have
$\At_x=\delta_\tau A_\xh$. Using \cite[Proposition 4.3.b]{kostant}, we get (sums  always run over positive real roots):
\begin{align*}
\At_x \lambda = \delta_\tau A_\xh \lambda &=
\delta_\tau \xh(\lambda) A_\xh + \delta_\tau \sum_{\a:\ \xh s_\alpha \lessdot \xh} \scal{\lambda,\a^\vee} A_{\xh s_\a} \\
&= \tau \xh(\lambda) \delta_\tau A_\xh + \sum_{\a:\ \xh s_\alpha \lessdot \xh} \scal{\lambda,\a^\vee} \delta_\tau A_{\xh s_\a} \\
&= x(\lambda) \At_x + \sum_{\a:\ \xh s_\alpha \lessdot \xh} \scal{\lambda,\a^\vee} \At_{xs_\a}\, .
\end{align*}
Since, by Definition \ref{defi:Ax}, the relation $\xh s_\alpha \lessdot \xh$ holds if and only if the relation
$xs_\alpha \lessdot x$ holds, we get the result.
\end{proof}

\subsection{Module over $\Aafft$}

We now define a natural module over $\Aafft$ which will be identified in the next section
with the homology of the adjoint affine Grassmmannian $\OKad$.

\begin{defi}
\label{defwafftm}
Let $\Waffm$ resp. $\Wafftm$ be the set of minimal length representatives of the quotient $\Waff/W$
resp. $\Wafft/W$.
By \cite[Lemma 3.3]{lam-shimozono}, $\Waffm$ is
the set of elements $w=ut_\lv$ such that $\lambda^\vee \leq 0$ and
$\forall i \in I,\scal{\lv,\alpha_i}=0 \implies u(\alpha_i) >0\, .$
\end{defi}

We generalize the characterization of $\Waffm$ as follows:

\begin{lemma}
\label{wafftm}
We have $ut_\lv \in \Wafftm$ if and only if $\lv \leq 0$ and for all $i$ in $I$ it holds
$$\scal{\lv,\alpha_i}=0 \implies u(\alpha_i) >0\, .$$
\end{lemma}
\begin{proof}
Recall that we have a length formula in $\Wafft$ similar to the one in $\Waff$:
$$
\ell(ut_\lv) = \sum_{\a \in R^+} \left | \scal{\lv,\a} + \chi(u(\a)<0) \right |\, ,
$$
where $\chi(\cP)=1$ if $\cP$ is true and $\chi(\cP)=0$ if $\cP$ is false. This is proved in
\cite[Corollary 3.13]{cmp4}. It follows that
$$\ell(ut_\lv s_i)-\ell(ut_\lv)=|\scal{\lv,-\a_i}+\chi(u(\a_i)>0)|-|\scal{\lv,\a_i}+\chi(u(\a_i)<0)|\, .$$
This is non-negative for all
$i$ in $I$ if and only if for all $i$, $\scal{\lv,\a_i} \leq 0$, and
$\scal{\lv,\alpha_i}=0$ implies $u(\alpha_i) >0$.
\end{proof}

\begin{defi}
\label{moduleM}
For each $w \in \Waffm$, we define a variable $\xi_w$ and we set
$$\displaystyle M = \bigoplus_{w \in \Waffm} S \cdot \xi_w.$$
Recall \cite[\S 6.2]{lam-shimozono} that we may define a left $\Aaff$-module structure on $M$ via:
$$
A_w \cdot \xi_u = \left \{
\begin{array}{l}
\xi_{wu} \mbox{ if } \ell(wu)=\ell(w)+\ell(u) \mbox{ and } wu \in \Waffm \, ,\\
0 \mbox{ otherwise}\, .
\end{array}
\right .
$$
As left $\Aaff$-module, we have an isomorphism
$$M \simeq \Aaff/J, \textrm{ where } J = \bigoplus_{w \not \in \Waffm} S \cdot A_w.$$
\end{defi}

Using Proposition \ref{prop:troisdef}, we define similarly a left ideal in $\Aafft$.

\begin{defi}
\label{defi:Jtilde}
Let $\displaystyle \Jt = \varphi_1(\Z[\Pv] \otimes_{\Z[\Qv]} J)
=\varphi_2(\Z[Z] \otimes_\Z J) = \bigoplus_{w \not \in \Wafftm} \At_w \, .$
\end{defi}

\begin{defi}
\label{defi:Mtilde}
We introduce the following three modules.
\begin{itemize}
 \item Let $\Mt_1$ be the $S$-module $\Z[\Pv] \otimes_{\Z[\Qv]} M$.
 \item Let $\Mt_2 = \Z[Z] \otimes_\Z M$. This is an $\Aafft$-module with the action given by
 $(\sigma \otimes a) \cdot (\tau \otimes \xi) = \sigma \tau \otimes \tau^{-1}(a) \cdot \xi$, for
 $\sigma \otimes a \in \Z[Z] \otimes_\Z \Aaff = \Aafft$.
 \item Let $\displaystyle \Mt_3 = \bigoplus_{w \in \Wafftm} S \cdot \xit_w$. This is an
 $\Aafft$-module with the action given by
 $$
 \At_w \cdot \xit_u = \left \{
\begin{array}{l}
\xit_{wu} \mbox{ if } \ell(wu)=\ell(w)+\ell(u) \mbox{ and } wu \in \Wafftm \, ,\\
0 \mbox{ otherwise}\, ,
\end{array}
\right .
$$
for $\displaystyle \At_w \in \Aafft = \bigoplus_{w \in \Wafft} S \cdot \At_w$.
\end{itemize}
\end{defi}

\begin{prop}
\label{3modules}
With the above definitions,
\begin{enumerate}
 \item Moding out by $\Jt$, the morphism $\varphi_1$ induces
 an $S$-module isomorphism $\Mt_1 \to \Aafft/\Jt$.
 \item Moding out by $\Jt$, the morphism $\varphi_2$ induces
 an $\Aafft$-module isomorphism $\Mt_2 \to \Aafft/\Jt$ (which is \emph{not} $S$-linear if we give
 $\Mt_2$ the tensor product $S$-module structure).
 \item The left $\Aafft$-modules $\Aafft/\Jt$ and $\Mt_3$ are isomorphic.
\end{enumerate}
\end{prop}
\begin{proof}
This proposition follows easily from Propositions \ref{product1} and
\ref{product3}.
\end{proof}

\begin{remark}
  The $\Aafft$-module structure induced on $\Mt_1$ by the isomorphism in Proposition \ref{3modules}.(1) can also be described via Proposition \ref{product2}.
\end{remark}

\begin{defi}
\label{defiMtilde}
The $\Aafft$-module defined by one of the above equivalent definitions will be denoted by $\Mt$.
\end{defi}

\section{Homology of the adjoint affine Grassmannian $\OKad$}
\label{section:homology}

In this section, we recall the adjoint affine Grassmannian $\OKad$,
we prove that the $\Aafft$-module $\Mt$ is isomorphic to the homology of $\OKad$, we define a ring structure on this module and study the compatibility of these two structures.

\subsection{Cohomology of the finite-dimensional flag manifold $G/B$}

Recall, see for example \cite[Chapter 11]{kumar}, that $H^*_T(G/B)$ has an $S$-basis $(\xi^w)_{w \in W} = (\sigma^B(w))_{w \in W}$ indexed by the Weyl group. The pull-back along the map $(G/B)^T \to G/B$ induces an inclusion
$$H^*_T(G/B) \to H_T^*((G/B)^T) = S^W.$$
Viewing $\xi^w$ as a function on $W$, Kumar \cite[11.1.6.(3)]{kumar} sets $d_{u,v} = \xi^u(v) = \scal{\xi^u,v}$ and $D = (d_{u,v})_{u,v \in W}$. If $(f^w)_{w \in W}$ is the basis of $S^W$ given by $\scal{f^u,v} = f^u(v) = \delta_{u,v}$, then we have $(\xi^u)_u = D(f^u)_u$.
Given the identification \cite[11.1.4(2)]{kumar}, we also have $\scal{f^u,\delta_v}=\delta_{u,v}$.

The dual of $H^*_T(G/B)$ is $H_*^T(G/B)$ and identifies as an $S$-module
with the $S$-subalgebra $\A$ of $\Aaff$ generated by $(A_w)_{w \in W}$:
\begin{equation}
\label{homology}
H_*^T(G/B) \simeq \bigoplus_{w \in W} S \cdot A_w
\end{equation}
Note that $(A_w)_{w \in W}$ is the dual basis to $(\xi^w)_{w \in W}$ \emph{i.e.} $\scal{\xi^u,A_v} = \xi^u(A_v) = \delta_{u,v}$ (see \cite[11.1.5]{kumar}, were $A_u$ is denoted by $x_u$). Over $F = {\rm Frac}(S)$ we also have the basis $(\delta_w)_{w \in W}$ for $H_*^T(G/B)$. Kumar, in \cite[11.1.2.(e)]{kumar}, describes the base change:
$$A_u = \sum_v c_{u,v} \delta_v$$
with $C = (c_{u,v})_{u,v \in W}$ a matrix with coefficients in $S$, in particular, we have $(A_v)_v = C(\delta_v)_v$. We have the following relation between the matrices $C$ and $D$.

\begin{fact}[See {\cite[11.1.7.(a)]{kumar}}]
\label{fact_base_change}
We have $D^{-1} = C^T$. Thus,
$$
\delta_v = \sum_{w \leq v} \xi^w(v) A_w\, .
$$
\end{fact}
\begin{proof}
In fact, from the identity $(A_v)_v=C(\delta_v)_v$, we deduce that
$(\delta_v)_v=C^{-1}(A_v)_v=D^T(A_v)$. Since $D^T_{v,w}=\xi^w(v)$ and the
matrix $D$ is triangular, we get the result.
\end{proof}

\noindent
Note also that an explicit formula for the coefficients $\xi^w(v)$ is known: see
\cite[Proposition 11.1.11]{kumar}.

\subsection{Affine Grassmannian and the Pontryagin ring structure}

Let $G$ be the simply-connected almost simple group associated to $\fg$, and let
$\Gad$ be the adjoint quotient of this group. Let $K$ resp. $\Kad$ be a maximal compact
subgroup in $G$ resp. $\Gad$.
Let $\OK$ resp. $\OKad$ be
the group of loops $l$ with values in $K$ resp. $\Kad$ such that $l(0)$ is the unit
element in $K$ resp. $\Kad$.
By a loop we mean a map $l:\mathbb{S}^1 \to K^{(\rm ad)}$ that extends to a meromorphic map
$\mathbb{D}^\circ \to G^{(\rm ad)}$, where $\mathbb{D}^\circ$ denotes the pointed disk. Moding out a loop by the center
of $K$ yields an inclusion
$\OK \subset \OKad$. The action of $T \cap K$
on $\OK$ resp. $\OKad$ is given by conjugation.

This implies that the equivariant homology of $\OK$ and
$\OKad$ have a natural structure of an algebra, given by the Pontryagin product which is also
$(T \cap K)$-equivariantly homotopy equivalent to the point-wise product of loops. In this section,
we will recall an algebraic model for $H_*^{T \cap K}(\OK)$ and give one for
$H_*^{T \cap K}(\OKad)$. In particular we will describe the ring structure as well as an $\Aafft$-module structure
on $H_*^{T \cap K}(\OKad)$ extending the ring structure and the $\Aaff$-module structure on $H_*^{T \cap K}(\OK)$.

\subsection{Geometry of fixed points in $\OKad$}

Since $K \to K^\ad$ is the universal cover of $K^\ad$, the connected components of $\OKad$ are isomorphic to $\OK$ and are indexed by $\pi_1(G^\ad) = \pi_1(K^\ad) = Z$. We now describe the $T$-fixed points in $\OKad$. We have, in the loop space picture
$$(\OKad)^T = \{ \psit_{t_\lv} : \mathbb{S}^1 \to K^\ad \ | \ \lv \in P^\vee\},$$
where $\psit_{t_\lv}(t) = \exp(2i\pi t\lv)$ is the loop induced by the one-parameter subgroup $\lv$ of $T^\ad$ (the maximal torus of $K^\ad$). For $\lv \in P^\vee$, let $[\lv]$ be its class in $P^\vee/Q^\vee = \pi_1(K^\ad)$ and denote by $\OK_{[\lv]}^\ad$ be the connected component of $\OKad$ containing $\psit_{t_\lv}$. We have
$$\OKad = \coprod_{[\lv] \in P^\vee/Q^\vee} \OK_{[\lv]}^\ad.$$
Let $m_\lv : \OK \to \OK_{[\lv]}^\ad$ be the left multiplication by $\psit_{t_\lv}$. Since $T$ and $\psit_{t_\lv}$ commute,
this is a $T$-equivariant isomorphism. Thus, $H_*^T(\OK_{[\lv]}^\ad) \simeq H_*^T(\OK)$.

\subsection{Reminder on $H_*^T(\OK)$}

Recall from \cite{kostant} that $\OK$ has a cellular decomposition whose cells are indexed by $\Waffm$.
This implies that, as $S$-module, we have
$$H_*^T(\OK) = \bigoplus_{w \in \Waffm} S \cdot \xi_w \simeq M.$$
Furthermore, according to \cite[(3.1) and (3.2)]{lam}, $\Aaff$ acts on $H_*^T(\cG/\cP)$ by
$$
A_v \cdot \xi_w =
\left \{ 
\begin{array}{l}
\xi_{vw} \mbox{ if } \ell(vw)=\ell(v)+\ell(w) \mbox{ and } vw \in \Waffm \\
0 \mbox{ otherwise}
\end{array}
\right.
$$
and
$\Aaff$ acts on $H^*_T(\cG/\cP)$ by
$$
A_v \cdot \xi^w =
\left \{ 
\begin{array}{l}
\xi^{vw} \mbox{ if } \ell(vw)=\ell(w)-\ell(w) \mbox{ and } vw \in \Waffm \\
0 \mbox{ otherwise}
\end{array}
\right.
$$

\subsection{$S$-algebra structure on $H_*^T(\OKad)$}

We use the $T \cap K$-equivariant homology of the $T \cap K$-space $\OKad$, where $T \cap K$ acts
on $\OKad$ via $T \cap K \to (T \cap K)^{\rm ad} \to \Gad$.
The inclusion $T \cap K \to T$ induces an isomorphism in equivariant
cohomology $H_T^*(pt) \to H_{T \cap K}^*(pt)$.
Note that we have $H_{T \cap K}(\OK) \simeq H_{T \cap K}(\cG/\cP) \simeq H_{T}(\cG/\cP)$, where $\cG / \cP$ is the affine Grassmannian.
Abusing notations slightly, we will
denote in the following $H_{T \cap K}(\OK)$ simply by $H_T(\OK)$, and similarly for $H_T(\OKad)$.
The $T$-equivariant cohomology of the point is the symmetric algebra on $P$, namely $S$, see \cite[p.5]{brion},
so that the homology $H_*^T(\OKad)$ will be an $S$-module and even an $S$-algebra. We are not considering $T^{\rm ad}$-equivariant homology. 

\begin{prop}
  \label{prop-HOKad-algebre}
As $S$-algebras, we have: $H_*^T(\OKad) \simeq S[P^\vee] \otimes_{S[\Qv]} H_*^T(\OK).$
\end{prop}

\begin{proof}
  We have the following inclusions that are compatible with pointwise multiplication and $T$-equivariant inducing $S$-algebra morphisms
  $$\xymatrix{\OK^T \ar[r] \ar[d] & \OK \ar[d] \\
    (\OKad)^T \ar[r] & \OKad,} \qquad
\xymatrix{H_*^T(\OK^T) \ar[r] \ar[d] & H_*^T(\OK) \ar[d] \\
    H_*^T((\OKad)^T) \ar[r] & H_*^T(\OKad).}$$
Recall that we have bijections $(\OKad)^T \simeq P^\vee$ and $\OK^T \simeq Q^\vee$ that are group homomorphisms since $\psit_{t_\lv}\psit_{t_\mv} = \psit_{t_{\lv + \mv}}$. We thus have $H_*^T((\OKad)^T) \simeq S[P^\vee]$ and $H_*^T(\OK^T) \simeq S[Q^\vee]$. In particular, the above diagram induces an $S$-algebra morphism $S[P^\vee] \otimes_{S[Q^\vee]} H_*^T(\OK) \to H_*^T(\OKad)$. The restriction of this map to $\psit_{t_\lv} \otimes H_*^T(\OK^T) \to H_*^T((\OKad)^T)$ is the multiplication $m_\lv$. The above decomposition of $\OKad$ in connected components gives an isomorphism of $S$-modules
$$H_*^T(\OKad) = \bigoplus_{[\lv] \in P^\vee/Q^\vee} H_*^T(\OK_{[\lv]}^\ad)$$
proving that the map $S[P^\vee] \otimes_{S[Q^\vee]} H_*^T(\OK) \to H_*^T(\OKad)$ is surjective.

To prove injectivity, first note that, since $H_*^T(\OK)$ is a free $S$-module and $S[P^\vee]$ is a free $S[Q^\vee]$-module, the $S$-module $S[P^\vee] \otimes_{S[Q^\vee]} H_*^T(\OK)$ is free. We therefore only need to prove the injectivity of the map after base extension to $F = \Frac(S)$ the field of fractions of $S$. Now recall the following general result (see \cite[C.8 Theorem]{kumar}: on the level of $T$-equivariant cohomology we have isomorphisms $H^*_T(\OK) \otimes_S F \simeq H^*_T(\OK^T) \otimes_S F$ and $H^*_T(\OKad) \otimes_S F \simeq H^*_T((\OKad)^T)\otimes_S F$. This induces isomorphisms in $T$-equivariant homology:
$$H_*^T(\OK^T) \otimes_S F \simeq H_*^T(\OK) \otimes_S F 
\textrm{ and }
H_*^T((\OKad)^T) \otimes_S F \simeq H_*^T(\OKad) \otimes_S F.$$
After base change to $F$, since $H_*^T(\OK^T) \simeq F[Q^\vee]$ and $H_*^T((\OKad)^T) \otimes_S F \simeq F[P^\vee]$, our map is given by
$$F[P^\vee] \otimes_{F[Q^\vee]} F[Q^\vee] \to F[P^\vee]$$
and is therefore injective.
\end{proof}

Recall that, as $S$-module, we have an isomorphism $H_*^T(\OK) = M$. In particular the above results identifies $H_*^T(\OKad)$ with the $S$-module $\Mt$ of Definition \ref{defiMtilde}:

\begin{cor}
As $S$-modules, we have: $H_*^T(\OKad) \simeq \Mt$.
\end{cor}

\begin{cor}
  The exists an $\Aafft$-module structure on $H_*^T(\OKad)$ compatible with the $\Aaff$-module structure on $H_*^T(\OK)$. Furthermore, for this structure, we have an isomorphism of $\Aafft$-modules
$H_*^T(\OKad) \simeq \Mt.$
\end{cor}

\begin{proof}
We define the $\Aafft$-module structure on $H_*^T(\OKad)$.
Since we have the isomorphism of $S$-modules $H_*^T(\OKad) \simeq \Mt_1 =S[P^\vee] \otimes_{S[Q^\vee]} M$, we may extend the $\Aaff$-module structure on $M$ to the $\Aafft-$module structure $\Mt$.
\end{proof}

\begin{remark}
The above $\Aafft$-module structure on $H_*^T(\OKad)$ also has a geometric description, see \cite[Proposition 3.3]{cmp4}.
\end{remark}

\begin{remark}
  \label{err1}
The above result shows that our claim on \cite[Page 12]{cmp4} that $\HT(\OKad)$ is the tensor product \emph{ring} $\Z[Z] \otimes_\Z \HT(\OK)$ is wrong: by localization $\HT(\OKad)$ is a subring of $F[\Pv]$ and this Laurent polynomial algebra contains no roots of unity, whereas $\Z[Z] \otimes_\Z \HT(\OK)$ does.
\end{remark}

Recall that $\Wafft$ can be embedded in $\Aafft$ via $w \mapsto \delta_w$. The induced action is denoted by $x \cdot \xit := \delta_x \cdot \xit$ for $x \in \Wafft$ and $\xit \in H_*^T(\OKad)$. 
\begin{cor}
  Let $w \in W$ and $\lv,\mv \in P^\vee$, we have 
  $$wt_\lv \cdot \psit_{t_\mv} := \delta_{wt_\lv} \cdot \psit_{t_\mv} = \psit_{t_{w(\lv + \mv)}}.$$
\end{cor}

\begin{proof}
  As already explained in the proof of Proposition \ref{prop-HOKad-algebre}, we have $\delta_{t_\lv} \cdot \psit_{t_\mv} = \psit_{t_{\lv+\mv}}$. We need to check that $\delta_w \cdot \psit_{t_\mv} = \psit_{t_{w(\mv)}}$. But our identification of $H_*^T(\OKad)$ with $\Mt$ identifies $\psit_{t_\mv}$ with $h_{\mv} \otimes 1$. Recall that $1 = [\delta] \in \Aafft/\Jt$, so that $h_\mv \otimes 1 = [\delta_{t_\mv}]$ and $\delta_w \cdot \psit_{t_\mv} = \delta_w \cdot [\delta_{t_\mv}] = [\delta_w \delta_{t_\mv}] = [\delta_w \delta_{t_\mv} \delta_{w^-1}]$ since $\delta_{w^{-1}} \in \Jt$. We get
  $$\delta_w \cdot \psit_{t_\mv} = [\delta_w \delta_{t_\mv} \delta_{w^-1}] = [\delta_{t_w(\mv)}] = \psit_{t_w(\mv)}$$
  proving the result.
\end{proof}

\subsection{Compatibility between the ring and the $\Aafft$-module structure}
The above description of $\HT(\OKad)$ as ring and as $\Aafft$-module is not enough for our purposes: we need to be able to multiply two classes of the
form $\sigma \otimes \xi_x$ and $\tau \otimes \xi_y$, see also Remark \ref{err1}. To this end, we recall the definition and properties of $\jad$ given in \cite[\S 3.3]{cmp4}.

\begin{prop}
\label{jad}
There is an $S$-algebra isomorphism $\jad:\HT(\OKad) \to Z_{\Aafft}(S)$. It satisfies:
\begin{enumerate}
 \item $\jad(\xi)\cdot \xi' = \xi \xi'$ for $\xi,\xi' \in \HT(\OKad)$ ;
 \item $\jad(\psit_{t_\lv}) = \delta_{t_\lv}$ for $\lv \in \Pv$.
\end{enumerate}
For $w \in \Wafftm$, $\jad(\xit_w)$ is characterized by the two following properties:
\begin{enumerate}[label=(\alph*)]
 \item $\jad(\xit_w)$ is congruent to $A_w$ modulo $\sum_{x \in W \setminus \{e\}} \Aafft \cdot A_x$ ;
 \item $\jad(\xit_w)$ belongs to $Z_{\Aafft}(S)$.
\end{enumerate}

\end{prop}

The map $\jad$ has the following equivariance property:

\begin{prop}
\label{jadequivariant}
Let $u \in W, \lv \in \Pv, \xit \in \HT(\OKad)$. Then
\begin{enumerate}
\item
$
\jad(ut_\lv \cdot \xit) = \delta_{t_{u(\lv)}} \delta_u \jad(\xit) \delta_{u^{-1}}
= \delta_{ut_\lv} \jad(\xit) \delta_{u^{-1}}$;
\item $\delta_{t_\lv}\jad(\xit) = \jad(\xit) \delta_{t_{\lv}}.$
\end{enumerate}
\end{prop}
\begin{proof}
(1) Let $s \in S$ be a scalar, we have:
\begin{itemize}
 \item $\jad(ut_\lv \cdot s\xit) = \jad(u(s) ut_\lv \cdot \xit)=u(s) \jad(u t_\lv \cdot \xit)$ ;
 \item $\delta_{t_{u(\lv)}} \delta_u \jad(s\xit) \delta_{u^{-1}}
 =\delta_{t_{u(\lv)}} \delta_u s \jad(\xit) \delta_{u^{-1}}
 =u(s)\delta_{t_{u(\lv)}} \delta_u \jad(\xit) \delta_{u^{-1}}$.
\end{itemize}
Thus, by semi-linearity, it is enough to prove the result for $\xit=\psit_{t_\mv}$. For $\xit=\psit_{t_\mv}$, we have
$\jad(ut_\lv \cdot \psit_{t_\mv})=\jad(u\cdot \psit_{t_{\lv+\mv}})
=\jad(\psit_{t_{u(\lv)+u(\mv)}}) = \delta_{t_{u(\lv)+u(\mv)}}.$
We also have $\delta_{t_{u(\lv)}} \delta_u \jad(\psit_{t_\mv}) \delta_{u^{-1}}
=\delta_{t_{u(\lv)}} \delta_u \delta_{t_\mv} \delta_{u^{-1}}
= \delta_{t_{u(\lv)+u(\mv)}}.$ Thus the result is proved.

(2) Both terms are $S$-linear so we only need to check this for $\xit  = \psit_{t_\mv}$ but we have $\delta_{t_\lv}\jad(\psit_{t_\mv}) = \delta_{t_\lv} \delta_{t_\mv} = \delta_{t_\mv} \delta_{t_\lv} = \jad(\psit_{t_\mv}) \delta_{t_{\lv}}$.
\end{proof}

\noindent
In particular, the previous Proposition allows computing $\jad$ in terms of $j$:

\begin{example}
Let $\tau_i=v_i t_{-\varpi_i} \in Z$ and let $\xi \in \HT(\OK)$. Then
$$
\jad(\tau_i \cdot \xi) = \delta_{\tau_i} \ j(\xi) \ \delta_{v_i^{-1}}\, .
$$
\end{example}

\medskip
We deduce a formula allowing reducing products in the homology of $\OKad$ to products in the homology of $\OK$:

\begin{cor}
\label{productaffine}
Let $\sigma = u t_\lv, \tau = v t_\mv$ be elements in $Z$. Let $\xit,\xit' \in \HT(\OKad)$. Then
  $$
(\sigma \cdot \xit) \times (\tau \cdot \xit') = \sigma \tau \cdot
(\psi_{\sigma,\tau} \times (v^{-1}_* \xit) \times (u^{-1}_* \xit'))\, ,
$$
where $\psi_{\sigma,\tau} = \psi_{u^{-1}(\mv)-\mv} = \psi_{v^{-1}(\lv)-\lv}$.
\end{cor}
\begin{proof}
Since $Z$ is abelian, $\sigma \tau = \tau \sigma$. We have
$\sigma \tau = ut_\lv vt_\mv = uv t_{v^{-1}(\lv)+\mv}$ and $\tau \sigma = vt_\mv ut_\lv = vu t_{u^{-1}(\mv)+\lv}$.
We get $v^{-1}(\lv)+\lv = u^{-1}(\mv)+\mv$ so $u^{-1}(\mv)-\mv = v^{-1}(\lv)-\lv$, so that $\psi_{\sigma,\tau}$ is well defined. We also get $uv=vu$.

Using Proposition \ref{jadequivariant}, we compute:
$$
\begin{array}{rcl}
(\sigma \cdot \xit) \times (\tau \cdot \xit') & = & \jad(\sigma \cdot \xit) \cdot (\tau \cdot \xit') \\
& = & \delta_u \delta_{t_\lv} \jad(\xit) \delta_{u^{-1}} \delta_v \delta_{t_\mv} \cdot \xit' \\
& = & \delta_u \delta_{t_\lv} \jad(\xit) \delta_v \delta_{u^{-1}} \delta_{t_\mv} \cdot \xit' \\
& = & \delta_u \delta_{t_\lv} \delta_v \jad(v^{-1} \cdot \xit) \delta_{t_{u^{-1}(\mv)}} \delta_{u^{-1}} \cdot \xit' \\
& = & \delta_u \delta_{t_\lv} \delta_v \delta_{t_{u^{-1}(\mv)}} \jad(v^{-1} \cdot \xit) \delta_{u^{-1}} \cdot \xit' \\
& = & \delta_u \delta_{t_\lv} \delta_v \delta_{t_\mv} \delta_{t_{u^{-1}(\mv)-\mv}} \jad(v^{-1} \cdot \xit) \delta_{u^{-1}} \cdot \xit' \\
& = & \sigma \tau \cdot (\psi_{\sigma,\tau} \times (v^{-1}_* \xit) \times (u^{-1}_* \xit'))\, .
\end{array}
$$
\end{proof}

\begin{remark}
In \cite[p.12]{cmp4}, it is claimed that $\HT(\OKad)$ is the tensor product ring $\Z[Z] \otimes_\Z \HT(\OK)$. As explained in Remark \ref{err1}, this is not true. However as the next corollary shows, this is true in the non equivariant homology.
\end{remark}

\begin{cor}
\label{productnonequivariant}
In non equivariant homology, let $\xit,\xit' \in H_*(\OKad)$ and $\sigma,\tau \in Z$, then
$$
(\sigma \cdot \xit) \times (\tau \cdot \xit') = \sigma \tau \cdot (\xit \times \xit') \, .
$$
\end{cor}
\begin{proof}
Indeed, push-forwards $u^{-1}_*$ and $v^{-1}_*$ are trivial in non equivariant homology. Moreover, the equivariant
classes $\psit_\lv$ restrict to the class of a point in $H_*(\OKad)$, which is the unit in $H_*(\OKad)$.
\end{proof}

\subsection{Translations modulo $P$}

We use \cite[Lemma 10.1]{lam-shimozono} and \cite[Corollary 3.15]{cmp4} as a definition:
\begin{equation}
\label{wpaff}
\Wpaff = \left \{ ut_\nv |\, \forall \gamma \in R_P^+,
\scal{\nv,\gamma} =
\left \{
\begin{array}{cl}
0 & \mbox{if } u(\gamma) >0 \\
-1 & \mbox{if } u(\gamma) <0
\end{array}
\right .
\right \}\, 
\end{equation}
\begin{equation}
\label{wpafft}
\Wpafft = \left \{ ut_\nv \in \Wafft \ \Big| \ \forall \gamma \in R_P^+,
\scal{\nv,\gamma} =
\left \{
\begin{array}{cl}
0 & \mbox{if } u(\gamma) >0 \\
-1 & \mbox{if } u(\gamma) <0
\end{array}
\right .
\right \}\, .
\end{equation}
Following \cite[\S 10.2 and 10.3]{lam-shimozono}, we also define
$(W_P)_{\rm aff} = \{wt_\lv \, | \, w \in W_P, \lv \in \Qvp \}$.
Recall, from \cite[Section 3.4]{cmp4} that any element $w \in \Wafft$ can be uniquely factorized as $w_1w_2$ with
$w_1 \in \Wpafft$ and $w_2 \in (W_P)_{\rm aff}$ and $\ell(w) = \ell(w_1) + \ell(w_2)$. We denote $w_1=\pi_P(w)$.
Thus $\Wpafft$ is a set of representatives for the quotient
$\Wafft/(W_P)_{\rm aff}$ which will be relevant for Peterson's isomorphism
(\ref{peterson}).

Following \cite[Section 10.4]{lam-shimozono} and \cite[Section 3.4]{cmp4}, define the ideals $J_P \subset M$ and $\Jt_P \subset \Mt$ as follows:
$$J_P = \sum_{x \in \Waffm \setminus \Wpaff} S \xi_x \textrm{ and } \Jt_P = \sum_{x \in \Wafftm \setminus \Wpafft} S \xit_x.$$

\medskip

The following result corrects \cite[Proposition 3.16]{cmp4} which used the wrong product structure, see Remark \ref{err1}.

\begin{prop}
\label{productpip}
Let $x \in \Wafftm \cap \Wpafft$ and let $\nv \in P^\vee_-$.
Then $x \pi_P(t_\nv) \in \Wafftm \cap \Wpafft$.
Let us write as usual $x=\sigma \xh$ and $\pi_P(t_\nv) = \tau \widehat{\pi_P(t_\nv)}$ with
$\sigma=ut_\lv,\tau=vt_\mv$. Then
$$
(v^{-1}_* \xi_\xh) \times (u^{-1}_* \xi_{\widehat{\pi_P(t_\nv)}}) = \psi_{\sigma,\tau}^{-1} \xi_{\widehat{x \pi_P(t_\nv)}} \textrm{ modulo } J_P.
$$
\end{prop}
\begin{proof}
The proof follows the arguments in \cite{cmp4}. In particular, we get
$$
(\sigma \cdot \xi_\xh) \times (\tau \cdot \xi_{\widehat{\pi_P(t_\nv)}})
= \sigma \tau \cdot \xi_{\widehat{x \pi_P(t_\nv)}} \textrm{ modulo } \Jt_P.
$$
Using the correct product formula given in Corollary \ref{productaffine}, the left hand side is
$$
\sigma \tau \cdot ( \psi_{\sigma,\tau} ( v^{-1}_* \xi_\xh) \times (u^{-1}_* \xi_{\widehat{\pi_P(t_\nv)}}))\, .
$$
This proves the result since $\Jt_P \cap M = J_P$ (as $\sigma\xh \in \Wafftm \Leftrightarrow \xh \in \Waffm$).
\end{proof}

In particular, the case $P=B$ yields:

\begin{cor}
\label{magyar}
Let $x \in \Wafftm$ and let $\nv \in P^\vee_-$.
Then $x t_\nv \in \Wafftm$.
Let us write as usual $x=\sigma \xh$ and $t_\nv = \tau \widehat{t_\nv}$ with
$\sigma=ut_\lv,\tau=vt_\mv$. Then
$$
(v^{-1}_* \xi_\xh) \times (u^{-1}_* \xi_{\widehat{t_\nv}}) = \psi_{\sigma,\tau}^{-1} \xi_{\widehat{x t_\nv}} \, .
$$
\end{cor}

\section{Affine symmetries}

In this section, we correct \cite[Section 3.5]{cmp4}, see Remark \ref{err1} using the correct product formula given in Corollary \ref{productaffine} and Proposition \ref{productpip}. In particular we prove that the formulas given in \cite{cmp4} are correct in the non equivariant setting.

\subsection{Peterson's isomorphism}
\label{subsection-peterson}

Proposition \ref{productpip} is our needed result in the equivariant homology of the affine
Grassmannian. Translating this formula in the quantum cohomology of $G/P$, we prove our
main theorem. We use Peterson's isomorphism proved in \cite{lam-shimozono} to relate
$\HT(\OK)$ and $\QHT(G/P)$.

Let $\eta_P:\Qv \to \Qvp$ be the projection on the coroot subspace generated by simple roots
$\alpha_i$ with $\alpha_i \not \in R_P$.
Peterson's isomorphism is the map
\begin{equation}
\label{peterson}
\fonction{\psi_P}{\HT(\OK)_P}{\Z[\Qvp] \otimes_\Z H_T^*(G/P)}
{\xi_{w\pi_P(t_\lv)} \xi_{\pi_P(t_\mv)}^{-1}}{q_{\eta_P(\lv-\mv)} \sigma^P(w)}
\end{equation}
where $w \in W^P$ and $\lv, \mv \in \Qvm$ with $\Qvm$ the set of antidominant elements in $Q^\vee$.
\begin{remark}
  \begin{enumerate} In the above statement we have:
  \item The space $\HT(\OK)_P$ is a quotient and a localization of $\HT(\OK)$ defined in \cite[\S 2.2]{cmp4}. The family $\{\sigma^P(w),w \in W/W_P\}$ is the Schubert base of $H_T^*(G/P)$, and the element in $\Z[\Qvp]$ corresponding to $\nv \in \Qvp$ is denoted by $q_{\nv}$. We have for $\nv \in \Qvm$ the formula
    $$\deg(q_\nv) = \sum_{\alpha \in R^+ \setminus R_P^+} \scal{\nv,\alpha} = - \ell(t_\nv).$$
      \item This isomorphism is graded. In fact, for very negative coweights $\lv,\mv$, the element $\xi_{w\pi_P(t_\lv)} \xi_{\pi_P(t_\mv)}^{-1}$ has homological degree $\ell(\pi_P(t_\lv)) -\ell(w) - \ell (\pi_P(t_\mv))$, by \cite[Lemma 3.3]{lam-shimozono}. On the other hand, in quantum cohomology, the element $q_{\eta_P(\lv-\mv)} \sigma^P(w)$ has degree
$-\ell(\pi_P(t_\lv)) + \ell (\pi_P(t_\mv))+\ell(w)$.
  \end{enumerate}
\end{remark}

\subsection{A Weyl group action on $\QH^*_T(G/P)$.}

In this subsection, we define an action of the Weyl group on $H^*_T(G/P)$ and on $\QH^*_T(G/P)$ by left translation. We will prove the compatibility of this action with Peterson's isomorphism in the next subsection. Since this action is different from the action defined in Kumar \cite[11.3.4]{kumar} we define it carefully. We start with the action on $G/B$ and then deal with the general situation for $G/P$.

We define an algebraic and a geometric action of the Weyl group $W$ on $H^*_T(G/B)$. We then prove that these actions coincide.

Let $n \in G$ be in the normalizer of $T$ and let $w$ be the corresponding element of the Weyl group. Define the left action $L_n:G/B \to G/B$ by left
multiplication: $L_n\cdot [x]=[n^{-1}x]$. This action is $T$-equivariant if we consider the
$w$-twisted action of $T$ on $G/B$ given by $t \cdot [x]=[w(t)x]$. It therefore induces a
$w$-semilinear map $H^*_T(G/B) \to H^*_T(G/B)$, denoted $L_n^*$:
$L_n^*(s\xi)=w(s)L_n^*(\xi)$ for $s \in S$ and $\xi \in H^*_T(G/B)$.
\begin{fact}
\label{waction}
The above action $L_n^*$ satisfies the following properties:
\begin{enumerate}
\item $L_n^*$ depends on $w$ and not on $n$ itself; it will be denoted by $w^*$ in the sequel.
\item Via the inclusion $H^*_T(G/P) \subset H^*_T(G/B)$ given by pulling back the projection
$G/B \to G/P$, we have $w^* H^*_T(G/P) \subset H^*_T(G/P)$.
\item The induced action of $w^*$ on the non equivariant cohomology $H^*(G/B)$ is trivial.
\end{enumerate}
\end{fact}
\begin{proof}
(1) Let $N$ denote the normalizer of $T$. The map $N \times G/B \to G/B,(n,[x]) \mapsto L_{n} \cdot [x]$
is continuous and therefore for $\xi \in H^*_T(G/B)$, the map $N \mapsto L_{n}^* \xi$ is locally
constant.

(2) For $n$ in the normalizer of $T$, we have a commutative diagram:
$$
\begin{tikzcd}
G/B \arrow[r,"L_n^B"] \arrow[d] & G/B \arrow[d] \\
G/P \arrow[r,"L_n^P"] & G/P
\end{tikzcd}
$$
Here we made a difference between the action of $n$ on $G/B$ and $G/P$ using superscripts.
It follows that for $\xi \in H^*_T(G/P)$, we have
$w^* \xi = (L_n^B)^* \xi = (L_n^P)^* \xi \in H^*_T(G/P)$.

(3) For $g \in G$, we can consider the action of left translation $L_g^*$ on non equivariant
cohomology $H^*(G/B)$. By the same argument as in (1), this action is trivial. In particular, for $g=n$ in $N$, we obtain that the action $L_n^*$ on non equivariant cohomology is trivial.
\end{proof}

Recall that $W$ can be embedded in $\Aaff$ via $v \mapsto \delta_v$.
\begin{defi}
Let $w \in W$.
Consider $H^*_T(G/B)$ as the dual of $H_*^T(G/B) \subset \Aaff$ and set
$$(w \bullet f)(x) = f(\delta_{w^{-1}} x) \textrm{ for $x \in H_*^T(G/B) = \A$}.$$
\end{defi}

\begin{prop}
For $f \in H^*_T(G/B)$ and $w \in W$, we have $w \bullet f = w^*f$.
\end{prop}

\begin{proof}
Using ${\rm Frac}(S)$-linearity, we only need to compare these actions on the elements $\xi^v$. We have
$(w \bullet \xi^v)(\delta_u) = \xi^v(\delta_{w^{-1}} \delta_u) = \xi^v(\delta_{w^{-1}u})
= \delta_{v,w^{-1}u} = \xi^v(w^{-1}u) = (w^*\xi^v)(u)$, proving the result.
\end{proof}

\begin{cor}
\label{coro-action-sP}
Let $\alpha$ be a simple root and $w \in W^P$. We have
$$(s_\alpha)^*\sigma^P(w) = \left\{
\begin{array}{ll}
\sigma^P(w) & \textrm{if } s_\alpha w > w ; \\
\sigma^P(w) - \alpha \sigma^P(s_\alpha w) & \textrm{if } s_\alpha w < w. \\
\end{array}\right.$$
\end{cor}

\begin{proof}
We compute $((s_\alpha)^*\sigma^P(w))(A_u)  = \sigma^P(w)(\delta_{s_\alpha}A_u) = \sigma^P(w)((1 - \alpha A_\alpha)A_u) = \sigma^P(w)(A_u) - \alpha \sigma^P(w)(A_\alpha A_u)$. Now we have 
$$A_\alpha A_u = \left\{
\begin{array}{ll}
0 & \textrm{if } s_\alpha u < u \\
A_{s_\alpha u} & \textrm{if } s_\alpha u > u. \\
\end{array}\right.$$
Since $\sigma^P(w)(A_v) = \delta_{v,w}$, we get
$$((s_\alpha)^*\sigma^P(w))(A_u) = \left\{
\begin{array}{ll}
\delta_{u,w} & \textrm{if } s_\alpha u < u \\
\delta_{u,w} - \alpha \delta_{s_\alpha u,w} & \textrm{if } s_\alpha u > u. \\
\end{array}\right.$$
This in turn gives the result.
\end{proof}

\begin{remark}
\begin{enumerate}
    \item Note that, for $\alpha$ simple, the two conditions $w \in W^P$ and $s_\alpha w < w$ imply the inclusion $s_\alpha w \in W^P$ since the inversion set of $s_\alpha w$ is contained in the inversion set of $w$. In particular, in the second case of the above formula, the class $\sigma^P(s_\alpha w)$ is well defined.
\item This formula also shows that the action $w^*$ is trivial in the non equivariant setting (indeed, in that case wet set $\alpha = 0$). 
\end{enumerate}
\end{remark}

The action $w^*$ is extended to $\QH^*_T(G/P)$ by linearity on quantum parameters.

\subsection{Compatibility of Peterson's isomorphism}

In this subsection we prove that Peterson's isomorphism is compatible with the actions $u_*$ in homology and $u^*$ in cohomology. We start with a useful lemma.

\begin{lemma}
\label{beta}
Let $w \in W^P$ and let $\lv \in \Qvm$ be such that $x = w\pi_P(t_\lv) \in \Wpaff$. Write $x = w \pi_P(t_\lv) = v t_\mv$ with $v \in W$ and $\mv \in Q^\vee$.

Let $\alpha$ be a simple root and let $\beta = w^{-1}(\alpha)$, $\beta' = v^{-1}(\alpha)$.
\begin{enumerate}
\item We have $w^{-1}v \in W_P$ and $\mv \in Q^\vee_-$.
\item We have $\beta \in R_P \Longleftrightarrow \beta' \in R_P$.
\item We have $s_\alpha x \in \Waffm \Longleftrightarrow \scal{\mv,\beta} \neq 0$.
  \item We have $s_\alpha x \in \Wpaff \Longleftrightarrow \beta \not \in R_P \Longleftrightarrow s_\alpha w \in W^P$.
  \item We have the equivalence:
    $$\left( s_\alpha x \in \Waffm \cap \Wpaff \textrm{ and } \ell(s_\alpha x) > \ell(x) \right) \Longleftrightarrow \left( s_\alpha w \in W^P \textrm{ and } \ell(s_\alpha w) < \ell(w)\right).$$
\end{enumerate}
\end{lemma}

\begin{proof}
  (1) By \cite[Lemma 10.7]{lam-shimozono}, we have $\pi_P(t_\lv) = u t_\mv$ with $u \in W_P$. This give $w^{-1}v = u \in W_P$. Since $vt_\mv \in \Waffm$ we have $\mv \in Q^\vee_-$.

  (2) Since $u = w^{-1}v \in W_P$ and $\beta = u(\beta')$, we have $\beta \in R_P \Leftrightarrow \beta' \in R_P$.

  (3) We have $v t_\mv \in \Waffm$ therefore $\mv \in Q^\vee_-$ and for $\gamma >0$, we have the implication $(\scal{\mv,\gamma} = 0 \Rightarrow v(\gamma) > 0)$.
  The condition $s_\alpha vt_\mu \in \Waffm$ is thus equivalent to $(\scal{\mv,\gamma} = 0 \Rightarrow s_\alpha v(\gamma) > 0)$ for $\gamma >0$. But since for $\gamma = \beta'$, the roots $v(\gamma)$ and $s_\alpha v(\gamma)$ have opposite signs, the condition $s_\alpha v t_\mv \in \Waffm$ is equivalent to $\scal{\mv,\beta'} = 0$.

  (4) We have $v t_\mv \in \Wpaff$ therefore, for $\gamma \in R_P^+$, we have the equivalences $(\scal{\mv,\gamma} = 0 \Leftrightarrow v(\gamma) > 0)$ and $(\scal{\mv,\gamma} = -1 \Leftrightarrow v(\gamma) < 0)$.
  The condition $s_\alpha vt_\mu \in \Wpaff$ is equivalent to having the equivalences $(\scal{\mv,\gamma} = 0 \Leftrightarrow s_\alpha v(\gamma) > 0)$ and $(\scal{\mv,\gamma} = -1 \Leftrightarrow s_\alpha v(\gamma) < 0)$. Since for $\gamma = \beta'$, the roots $v(\gamma)$ and $s_\alpha v(\gamma)$ have opposite signs, the last equivalences occur if and only if $\beta' \not\in R_P$. This in turn is equivalent to $\beta \not\in R_P$ by (2).

  For the last equivalence, note that by definition, we have $(s_\alpha w \in W^P \Leftrightarrow s_\alpha w(R_P^+) \subset R^+)$. Since  $w \in W^P$, we have $w(R_P^+) \subset R^+$. Since the inversion sets of $w$ and $s_\alpha w$ only differ by $\beta$ (or its opposite, depending on the sign of $\beta$) we get the last equivalence. 

  (5) Note that we have the equivalence $(\ell(s_\alpha w) < \ell(w) \Leftrightarrow \beta < 0)$. We therefore need to prove that the left hand side of the equivalence is equivalent to $\beta \not\in R_P$ and $\beta < 0$. Note that since $w \in W^P$, this is equivalent to $\beta < 0$.

  First assume that $s_\alpha x \in \Waffm \cap \Wpaff$ and $\ell(s_\alpha x) > \ell(x)$. By \cite[Lemma 3.3]{lam-shimozono}, since $v t_\mv,s_\alpha v t_\mv \in Waffm$, we have $\ell(s_\alpha x) = \ell(t_\mv) - \ell(s_\alpha v)$ and $\ell(x) = \ell(t_\mv) - \ell(v)$. In particular, we have $\ell(s_\alpha v) < \ell(v)$, thus $\beta' < 0$. Since $s_\alpha x \in \Wpaff$, we also have $\beta \not\in R_P$ thus $\beta' \not\in R_P$. Now, since $u \in W_P$, this implies  $\beta < 0$.

  Conversely, assume $\beta < 0$. By the above arguments, this implies $\beta \not\in R_P$ and thus $s_\alpha x \in \Wpaff$. This also implies $\beta' \not\in R_P$ and since $u \in W_P$ and $\beta' = u^{-1}(\beta)$, we get $\beta' < 0$. Since $v t_\mu \in \Waffm$, $-\beta' > 0$ and $v(-\beta') = - \alpha < 0$, we must have $\scal{\mv,\beta'} \neq 0$ and by (3), this implies $s_\alpha x \in \Waffm$.
\end{proof}
 
We have the following equivariance property of $\psi_P$.

\begin{prop}
\label{psi-equivariant}
For $\xi \in \HT(\OK)_P$, we have $\psi_P(u_* \xi) = u^* \xi$.
\end{prop}
\begin{proof}
We may assume that $u=s_i$, with $\alpha_i$ a simple root. Then $u_* \xi_{w \pi_P(t_\lv)} = \delta_u \cdot \xi_{w \pi_P(t_\lv)} = (1-\alpha_i A_i) \cdot \xi_{w \pi_P(t_\lv)}$. If $\ell(s_iw\pi_P(t_\lv)) > \ell(w \pi_P(t_\lv))$ and $s_iw \pi_P(t_\lv) \in \Wpaff \cap \Waffm$, then this is equal to $\xi_{w \pi_P(t_\lv)} - \alpha_i \xi_{s_iw\pi_P(t_\lv)}$. Otherwise, this is equal to $\xi_{w \pi_P(t_\lv)}$.

The action ${s_i}^* \sigma^P(w)$ is computed in Corollary \ref{coro-action-sP}.
If $\ell(s_iw) < \ell(w)$ and $s_iw \in W^P$, then this is equal to
$\sigma^P(w)-\alpha_i \sigma^P(s_iw)$. Otherwise, this is equal to $\sigma^P(w)$.

Let $\beta=w^{-1}(\alpha_i)$. The condition $s_iw \in W^P$ and $\ell(s_iw) < \ell(w)$ is equivalent to the condition $\ell(s_iw\pi_P(t_\lv)) > \ell(w \pi_P(t_\lv))$ and $s_iw \pi_P(t_\lv) \in \Wpaff \cap \Waffm$ by Lemma \ref{beta}.(5). This proves the result. 
\end{proof}

\subsection{The result}

We now prove our main result. For $i$ a cominuscule node, \emph{i.e.} such that $\varpi_i^\vee$
is a minuscule coweight, we let $v_i$ be the smallest element in $W$ such that
$v_i(\varpi_i^\vee) = w_0(\varpi_i^\vee)$ ($w_0$ is the longest element in $W$). The coweight $v_i(\varpi_i^\vee) = w_0(\varpi_i^\vee)$ is the opposite of a fundamental coweight: there exists $f(i) \in I$ such that $v_i(\varpi_i^\vee) = - \varpi_{f(i)}^\vee$. Actually we have $\alpha_{f(i)} = -w_0(\alpha_i)$ and $v_{f(i)} = v_i^{-1}$.

\begin{thm}
\label{theo}
Let $i$ be a cominuscule node. In $QH^*_T(G/P)$ we have
$$
\sigma^P(v_i) \times v_{i}^*(\sigma^P(w)) = q_{\eta_P(\varpi_i^\vee - w^{-1}(\varpi_i^\vee))}
\sigma^P(v_iw) \, .
$$
\end{thm}

\begin{proof}
  Let $w \in W^P$, we have $\pi_P(w)=w$. Let $\varpi_i^\vee$ be the minuscule coweight associated to $i$ and let $\mv$ and $\nv$ be in $Q^\vee$ and dominant enough. As in \cite[\S 3.5]{cmp4}, we get
  $$\begin{array}{l}
    \pi_P(t_{-\varpi^\vee_i-\mu}) = \tau_i \pi_P(v_{f(i)}) \pi_P(t_{-(\varpi_i^\vee + \varpi_{f(i)}^\vee + \mu)} \textrm{ and } \\
    \pi_P(wt_{-\nu} t_{-\varpi^\vee_i-\mu}) = \tau_i \pi_P(v_{f(i)}w) \pi_P(t_{-(\varpi_i^\vee + w^{-1}(\varpi_{f(i)}^\vee) + \mu + \nu)}.\\
  \end{array}$$
For $\mu$ and $\nu$ dominant enough, the elements $wt_{-\nu}$, $t_{-\varpi_i^\vee - \mu}$ and $wt_{-\nu}t_{-\varpi_i^\vee - \mu}$ are in $\Wafftm$ and their image by $\pi_P$ are in $\Wpafft \cap \Wafftm$. We may therefore apply Proposition \ref{productpip} to the elements $wt_{-\nu}$ and $t_{-\varpi_i^\vee-\mu}$ to get:
$$
\begin{array}{cl}
& (v_{f(i)})_*\xi_{w\pi_P(t_{-\nv})} \times
\xi_{\pi_P(v_{f(i)})\pi_P(t_{-(\varpi_i^\vee+\varpi_{f(i)}^\vee+\mv)})} \\
\equiv &
\xi_{\pi_P(v_{f(i)}w)\pi_P(t_{-(\varpi_i^\vee+w^{-1}(\varpi_{f(i)}^\vee+\mv+\nv))})}\, ,
\end{array}
$$
where $\equiv$ means equality in $\HT(\OK)_P$ or equivalently equality modulo $J_P$. Applying Peterson's map (\ref{peterson}), we get thanks to Proposition \ref{psi-equivariant} the corresponding formula in the quantum cohomology ring:
$$
\begin{array}{cl}
& v_{f(i)}^*\sigma^P(w)q_{-\eta_P(\nv)}*\times \sigma^P(v_{f(i)})
q_{-\eta_P((\varpi_i^\vee+\varpi_{f(i)}^\vee+\mv))} \\
= & \sigma^P(v_{f(i)}w)q_{-\eta_P((\varpi_i^\vee+w^{-1}(\varpi_{f(i)}^\vee)+\mv+\nv))}\, ,
\end{array}
$$
hence finally:
$$v_{f(i)}^*\sigma^P(w)*\sigma^P(v_{f(i)}) =
q_{\eta_P(\varpi_{f(i)}^\vee-w^{-1}(\varpi_{f(i)}^\vee))}
\sigma^P(v_{f(i)}w).$$ This concludes the proof of the theorem.
\end{proof}

\begin{cor}
Let $i$ be such that $\varpi_i^\vee$ is a minuscule coweight. In $QH^*(G/P)$, we have
$$
\sigma^P(v_i) \times \sigma^P(w) = q_{\eta_P(\varpi_i^\vee - w^{-1}(\varpi_i^\vee))}
\sigma^P(v_iw) \, .
$$
\end{cor}

\begin{example}
\label{examplea1}
Let $G$ be of type $A_1$, so that $G/B=\P^1$. Let
$s$ be the non trivial element of $W$ and $\alpha$ the simple root.
We have $$\sigma^B(s)*(\sigma^B(s)-\alpha)=q\, .$$
\end{example}
\begin{proof}
Let $i$ be the unique node of the Dynkin diagram of $G$. Then $v_i=s$. To apply
Theorem \ref{theo}, we also set $w=s$. Let $x$ resp. $y$ be the $B$-stable resp. $B^-$-stable
point in $\P^1$. The class $\sigma^B(s)$ is the $T$-equivariant class of $x$, and
$v_i^* \sigma^B(s)$ is the $T$-equivariant class of $y$. Since $[x]-[y]=\alpha$, we have
$v_i^* \sigma^B(s)=\sigma^B(s)-\alpha$. Denoting $h=\sigma^B(s)$, the theorem yields
$h \times (h-\alpha)=q$, as claimed. Note that $h^2=q+\alpha h$ is also predicted eg by
\cite[Theorem 1]{mihalcea}.
\end{proof}

\section{Pieri formulas}
\label{pieri}

We now give another application of Proposition \ref{jad} to prove a formula for
$j(\xi_{\tau_i(v_i)})$, see Proposition \ref{prop_pieri}. This gives the multiplication in
$H^T_*(\OK)$ by the class $\xi_{\tau_i(v_i)}$. We hope in subsequent work to deduce Pieri formulas for the non-equivariant
multiplication by classes generating $H_*(\OK)$ in all classical types.

\medskip

We first provide a generalization of \cite[Proposition 5.4]{lam} to coweights. For $\mv \in P^\vee$, set $W_\mv = \scal{s_{\alpha_i} \ | \ i \in [1,r] \textrm{ and } \scal{\alpha_i,\mv} = 0} = \{ w \in W \ | \ w(\mv) = \mv \}$.
\begin{prop}
\label{jadtranslation}
Let $\mv \in \Pv$ be antidominant. Then
$$\jad(\xi_{t_\mv})=\sum_{w \in W/W_\mv} \At_{t_{w(\mv)}}\, .$$
\end{prop}
\begin{proof}
We follow the idea of proof given in \cite[Proposition 5.4]{lam}. Using
Lemma \ref{wafftm}, we see that for $w \in W/W_{\mv}$ non trivial, $t_{w(\mv)} \not \in \Wafftm$, so that
$A_{t_{w(\mv)}}$ belongs to the ideal $\sum_{x \in W \setminus \{e\}} \Aafft \cdot A_x$ of Proposition \ref{jad}.
Thus, using Proposition \ref{jad}, we only need to prove that $\sum_{w \in W/W_\mv} A_{t_{w(\mv)}} \in Z_{\Aafft}(S)$.

To prove that $c:=\sum_{w \in W/W_\mv} \At_{t_{w(\mv)}}$ centralizes $S$, or equivalently commutes with any $\lambda$
in $P$, we use
Proposition \ref{product4} to compute $\At_{t_\nv} \lambda$.
In this formula, the term $t_\nv(\lambda)$ is equal to $\lambda$ by (\ref{defaction})
in \S \ref{subsec:affineweyl}.
Let $\cP$ be the set of pairs
$(\nv,\beta)$ where $\nv \in W \cdot \mv$, $\beta$ is a positive real root, and $t_\nv s_\beta \lessdot t_\nv$.
We have:
\begin{equation}
\label{clambda}
c \lambda - \lambda c = \sum_{(\nv,\beta) \in \cP} \scal{\lambda,\beta^\vee} \At_{t_{\nv} s_\beta} \, ,
\end{equation}
so our concern now is to prove that this sum vanishes.

We consider the map $\iota:\cP \to \cP$ defined by $\iota(\nv,\beta)=(s_\beta(\nv),-t_\nv(\beta))$. Let
$(\nv,\beta) \in \cP$. We have
$$
\begin{array}{rclcl}
t_\nv s_\beta = t_\nv s_\beta t_{-\nv} t_\nv & = & s_{t_\nv(\beta)} t_\nv \\
& = & t_{s_{t_\nv(\beta)}(\nv)} s_{t_\nv(\beta)} & = & t_{s_\beta(\nv)} s_{-t_\nv(\beta)}\, ,
\end{array}
$$
where the last equality follows from (\ref{defaction}) and Lemma \ref{salphaepsilon}.
By the length formula in \cite[Corollary 3.13]{cmp4}, $\ell(t_{s_\beta(\nv)})=\ell(t_\nv)$ and by definition of $\cP$,
$\ell(t_\nv s_\beta) = \ell(t_\bv)-1$. Thus, $\ell(t_{s_\beta(\nv)} s_{-t_\nv(\beta)}) = \ell(t_{s_\beta(\nv)})-1$.
Moreover, by \cite[Proposition 4.4.6]{bb}, $t_\nv(\beta)<0$, which implies
$t_{s_\beta(\nv)} s_{-t_\nv(\beta)} \lessdot t_{s_\beta(\nv)}$ and $-t_\nv(\beta)>0$, so
$(s_\beta(\nv),-t_\nv(\beta)) \in \cP$ as claimed.

We also observe that $\scal{\lambda,-t_\nv(\beta)^\vee}=\scal{\lambda,-\beta^\vee}=-\scal{\lambda,\beta^\vee}$. Finally,
$$
-t_{s_\beta(\nv)}(-t_\nv(\beta)) = t_{s_\beta(\nv)} t_\nv(\beta) = s_\beta t_\nv s_\beta t_\nv(\beta)\, .
$$
One can check that this root is equal to $\beta$, so that $\iota$ is an involution and the terms in (\ref{clambda}) cancel
pairwise.
\end{proof}

We now prove some preliminary lemmas.

\begin{lemma}
\label{lemm_wi}
Let $i \in \Iaff$.
We have $\displaystyle \jad(\xit_{w_0(\varpi_i^\vee)}) = \delta_{\tau_i^{-1}} \sum_{w \leq_L v_i} A_{\tau_i(w)v_iw^{-1}}\, .$ 
\end{lemma}
\begin{proof}
Since $w_0(\varpi_i^\vee) \leq 0$ we may apply Proposition \ref{jadtranslation} and get
$$
\jad(\xi_{w_0(\varpi_i^\vee)}) = \sum_{\mv \in W \cdot w_0(\varpi_i^\vee)} A_{t_\mv} \, .
$$
Thus,
$$
\begin{array}{rcccl}
\jad(\xi_{w_0(\varpi_i^\vee)}) & = & \sum_{\mv \in W \cdot \varpi_i^\vee} A_{t_\mv} & = & \sum_{w \leq_L v_i} A_{wt_{\varpi_i^\vee}w^{-1}} \\
& = & \sum_{w \leq_L v_i} A_{w\tau_i^{-1}v_iw^{-1}} & = & \sum_{w \leq_L v_i} A_{\tau_i^{-1}\tau_i(w)v_iw^{-1}} \\
& = & \delta_{\tau_i^{-1}} \sum_{w \leq_L v_i} A_{\tau_i(w)v_iw^{-1}}\, .
\end{array}
$$
\end{proof}

\begin{lemma}
\label{lemm_commutation}
Let $s \in S$ and $i \in \Iaff$. We have
$$
\left ( \sum_{w \leq_L v_i} A_{\tau_i(w)v_iw^{-1}} \right ) s = \tau_i(s) \left ( \sum_{w \leq_L v_i} A_{\tau_i(w)v_iw^{-1}} \right ) \, .
$$
\end{lemma}
\begin{proof}
Let $i \in \Iaff$. Since $\jad(\xi_{w_0(\varpi_i^\vee)}) = \delta_{\tau_i^{-1}} \sum_{w \leq_L v_i} A_{\tau_i(w)v_iw^{-1}}$, we deduce that
$$\delta_{\tau_i^{-1}} \sum_{w \leq_L v_i} A_{\tau_i(w)v_iw^{-1}} \in Z_\Aaff(S) \, .$$
Let $s \in S$, we have:
$$
\delta_{\tau_i^{-1}} (\sum_{w \leq_L v_i} A_{\tau_i(w)v_iw^{-1}}) \, s = s \delta_{\tau_i^{-1}} \sum_{w \leq_L v_i} A_{\tau_i(w)v_iw^{-1}}
= \delta_{\tau_i^{-1}} \tau_i(s) \sum_{w \leq_L v_i} A_{\tau_i(w)v_iw^{-1}}\, ,
$$
which proves the lemma.
\end{proof}

\begin{prop}
\label{prop_pieri}
Let $i \in \Iaff$, let as above $v_i$ the maximal element in $W^{P_i}$ and $\tau_i$ the automorphism of the affine Dynkin diagram defined by $i$.
Then $\tau_i(v_i) \in \Waffm$ and we have:
$$
j(\xi_{\tau_i(v_i)}) = \sum_{w \leq_L v_i} \sum_{v \leq v_i^{-1}} \tau_i(\xi^v(v_i^{-1})) A_{\tau_i(w)v_iw^{-1}} A_v \, .
$$
\end{prop}
\begin{proof}
We first prove that $\tau_i(v_i) \in \Waffm$.
We know that $\tau_i(\alpha_i)=\alpha_0$.
Since $v_i \in W^{P_i}$, we have for $1 \leq j \leq n$ with $j \neq i$, $\ell(v_is_j)>\ell(v_i)$. Since $v_i \in W$, $\ell(v_is_0)>\ell(v_i)$. Applying
$\tau_i$, we deduce that for all $k>0$, $\ell(\tau_i(v_i)s_k)>\ell(\tau_i(v_i))$. Thus, $\tau_i(v_i) \in \Waffm$.

Moreover, we know that $\tau_i=v_it_{-\varpi_i^\vee}$. Therefore, $v_i=\tau_it_{\varpi_i^\vee}=t_{w_0(\varpi_i^\vee)} \tau_i$, so that
$t_{w_0(\varpi_i^\vee)} = v_i \tau_i^{-1} = \tau_i^{-1} \tau_i(v_i)$. By Proposition \ref{jadequivariant}, we deduce that
$\jad(\xi_{t_{w_0(\varpi_i^\vee)}}) = \delta_{\tau_i^{-1}} j(\xi_{\tau_i(v_i)}) \delta_{v_i}$.

By Lemma \ref{lemm_wi}, we deduce that
$$ 
\delta_{\tau_i^{-1}} j(\xi_{\tau_i(v_i)}) \delta_{v_i} = \delta_{\tau_i^{-1}} \sum_{w \leq_L v_i} A_{\tau_i(w)v_iw^{-1}} \, .
$$
Therefore, using Fact \ref{fact_base_change} and then Lemma \ref{lemm_commutation}, we find
$$
\begin{array}{rcl}
j(\xi_{\tau_i(v_i)}) & = & \sum_{w \leq_L v_i} A_{\tau_i(w)v_iw^{-1}} \delta_{v_i^{-1}} \\
& = & \sum_{w \leq_L v_i} A_{\tau_i(w)v_iw^{-1}} \sum_{v \leq v_i^{-1}} \xi^v(v_i^{-1}) A_v \\
& = & \sum_{w \leq_L v_i} \sum_{v \leq v_i^{-1}} \tau_i(\xi^v(v_i^{-1})) A_{\tau_i(w)v_iw^{-1}} A_v\, .
\end{array}
$$
\end{proof}

\begin{remark}
Let $x \in \Waffm$. In the non equivariant homology, we thus have 
$$
\xi_{\tau_i(v_i)} \cdot \xi_x = \sum \xi_{\tau_i(w)v_iw^{-1}x}\, ,
$$
where the sum is over $w \leq_L v_i$ such that $\ell(\tau_i(w)v_iw^{-1}x)=\ell(v_i)+\ell(x)$ and $\tau_i(w)v_iw^{-1}x \in \Waffm$.

By Corollary \ref{magyar}, we know that there is only one Schubert class in the product
$\xi_{\tau_i(v_i)} \cdot \xi_x$, from which we deduce that there is exactly one
$w \leq_L v_i$ such that $\ell(\tau_i(w)v_iw^{-1}x)=\ell(v_i)+\ell(x)$ and $\tau_i(w)v_iw^{-1}x \in \Waffm$.
\end{remark}

\begin{example}
\label{exam_pieri}
Let us assume we are in type $\widetilde A_3$ and let us write for short
$A_{210}$ instead of $A_{s_2s_1s_0}$ and similarly for $\xi_{210}$ and $\delta_{210}$.
Let $i=1$ so that $v_i=s_3s_2s_1$ and $\tau_i(v_i)=s_2s_1s_0$.
First we observe that 
$$
\begin{array}{rcl}
\delta_{\tau_i^{-1}} = \delta_{123} & = & (1-\alpha_1 A_1)(1-\alpha_2 A_2)(1-\alpha_3 A_3) \\
& = & 1 - \a_1 A_1 - (\a_1+\a_2)A_2-(\a_1+\a_2+\a_3)A_3 \\
& + & \a_1(\a_1+\a_2) A_{12} +  \a_1(\a_1+\a_2+\a_3) A_{13} + (\a_1+\a_2)(\a_1+\a_2+\a_3)A_{23} \\
& - & \a_1(\a_1+\a_2)(\a_1+\a_2+\a_3)A_{123}\, .
\end{array}
$$
Since $\tau_i(\a_1)=-\theta=-(\a_1+\a_2+\a_3)$, $\tau_i(\a_2)=\a_1$ and $\tau_i(\a_3)=\a_2$, we get:
$$
\begin{array}{rcl}
j(\xi_{210}) & = & A_{210} + A_{321} + A_{032} + A_{103} \\
& + & \a_3(A_{2103} + A_{3213} + A_{0323}) \\
& + & (\a_2+\a_3) (A_{2102} + A_{3212} A_{1032}) \\
& + & (\a_1+\a_2+\a_3) (A_{2101} + A_{0321} + A_{1031}) \\
& + & \a_3(\a_2+\a_3) (A_{21023} + A_{32123} + A_{10323}) \\
& + & \a_3(\a_1+\a_2+\a_3) (A_{21031} + A_{03231}) \\
& + & (\a_2+\a_3)(\a_1+\a_2+\a_3) (A_{21012} + A_{03212} + A_{10312}) \\
& + & \a_3(\a_2+\a_3)(\a_1+\a_2+\a_3) (A_{210123} + A_{032123} + A_{103123}) \, .
\end{array}
$$
\end{example}

\addtocontents{toc}{\protect\setcounter{tocdepth}{2}}
\providecommand{\bysame}{\leavevmode\hbox to3em{\hrulefill}\thinspace}
\providecommand{\MR}{\relax\ifhmode\unskip\space\fi MR }
\providecommand{\MRhref}[2]{%
  \href{http://www.ams.org/mathscinet-getitem?mr=#1}{#2}
}
\providecommand{\href}[2]{#2}


\end{document}